\documentclass[12pt]{amsart}
\usepackage[utf8]{inputenc}
\usepackage{amsfonts}
\usepackage{amssymb}
\usepackage{amsmath}
\usepackage[english]{babel}
\usepackage{amsthm}
\theoremstyle{plain}
\newtheorem{theorem}{Theorem}[section]

\newtheorem{lemma}[theorem]{Lemma}
\newtheorem{proposition}[theorem]{Proposition}
\theoremstyle{definition}
\newtheorem{definition}[theorem]{Definition}
\newtheorem{remark}[theorem]{Remark}

\usepackage{geometry}
\numberwithin{equation}{section} 
\usepackage{enumitem}
\usepackage{fancyhdr}

\usepackage{lipsum}

\newcommand\blfootnote[1]{%
  \begingroup
  \renewcommand\thefootnote{}\footnote{#1}%
  \addtocounter{footnote}{-1}%
  \endgroup
}

\usepackage[colorlinks, citecolor=blue,pagebackref,hypertexnames=false]{hyperref}

\newcounter{comcount}

\begin{document}

\title[Riesz transform and its reverse inequality on (QD) manifolds]{On the Riesz transform and its reverse inequality on manifolds with quadratically decaying curvature} 
\date{}
\author{DANGYANG HE}  
\address{Department of Mathematics, Sun Yat-sen University, Guangzhou, China \\
Formerly Department of Mathematics, Macquarie University, Sydney, Australia }
\email{hedy28@mail.sysu.edu.cn}

\begin{abstract}
We study Riesz and reverse Riesz inequalities on manifolds whose Ricci curvature decays quadratically. First, we refine existing results on the boundedness of the Riesz transform by establishing a Lorentz-type endpoint estimate. Next, we explore the relationship between the Riesz and reverse Riesz transforms, proving that the reverse Riesz, Hardy, and weighted Sobolev inequalities are essentially equivalent. Finally, we apply our methods to Grushin spaces, which exhibit a quadratic decay in 'Ricci curvature', verifying that the reverse inequality holds for all $p\in (1,\infty)$ and that the Riesz transform is bounded on $L^p$ for $p\in (1,n)$. Our approach relies on an asymptotic formula for the Riesz potential combined with an extension of the so-called harmonic annihilation method.
\end{abstract}

\maketitle

\tableofcontents

\blfootnote{$\textit{2020 Mathematics Subject Classification.}$ 42B20, 58J35, 35H10}

\blfootnote{$\textit{Keywords and Phrases.}$ Riesz transform, reverse Riesz inequality, Hardy's inequality, weighted Sobolev inequality, Ricci curvature, Grushin spaces.}

\section{Introduction}\label{section_1}

Let $M$ be a complete Riemannian manifold equipped with the Riemannian volume measure $d\mu$ (also denoted by $d\text{vol}$). Denote by $\nabla$ and $\Delta$ the associated Riemannian gradient and Laplace-Beltrami operator, respectively, which satisfy relationship:
\begin{equation*}
    \int g \Delta f d\mu = \int \nabla f \cdot \nabla g d\mu
\end{equation*}
for all $f,g\in C_c^\infty(M)$. One defines the Riesz transform by $\nabla \Delta^{-1/2}$. In \cite{S}, Strichartz initiated the question of what extent of the classical harmonic analysis of the Laplacian on Euclidean space can be extended to the Riemannian setting. A critical problem in this context is to determine the range of boundedness of $\nabla \Delta^{-1/2}$ on $L^p$ spaces. That is, for what $1\le p\le \infty$, one can establish inequality
\begin{equation}\tag{$\textrm{R}_p$}\label{R_p}
    \|\nabla f\|_p \le C  \|\Delta^{1/2}f\|_p,\quad \forall f\in C_c^\infty(M).
\end{equation}
Moreover, the study of \eqref{R_p} also leads to the concept of reverse Riesz inequality
\begin{equation}\tag{$\textrm{RR}_{p}$}\label{eq_RRp}
    \|\Delta^{1/2}f\|_{p} \le C \left\|\nabla f \right\|_{p} \quad \forall f\in C_c^\infty(M).
\end{equation}

For the past several decades, to study the Riesz and reverse Riesz inequalities, much attentions have been paid to the manifolds satisfying the so-called volume doubling condition: for any $x\in M$ and $R\ge r>0$, there exists some $\mu, C>0$ such that
\begin{equation}\tag{$\textrm{D}_\mu$}\label{Doubling}
    \frac{V(x,R)}{V(x,r)} \le C \left(\frac{R}{r}\right)^\mu,
\end{equation}
where $B(x,r)$ denotes the geodesic ball centered at $x$ with radius $r$ and $V(x,r) = \mu(B(x,r))$ denotes the volume of the ball. It is also known that by \cite[Lemma~2.10]{GS_2005}, the doubling condition \eqref{Doubling} implies the following reverse doubling property
\begin{align}\tag{$\textrm{RD}_\nu$}\label{RD}
    C\left(\frac{R}{r}\right)^\nu \le \frac{V(x,R)}{V(x,r)},
\end{align}
for some $0<\nu\le \mu$. 

Another frequently used condition in this context is the Poincaré inequality. Recall that one says $M$ satisfies $L^q\textit{-}$Poincaré inequality $(1\le q<\infty)$ if for any ball $B$ with radius $r>0$,
\begin{align}\tag{$\textrm{P}_q$}\label{P_q}
    \int_B |f-f_B|^q d\mu \le C r^q \int_B |\nabla f|^q d\mu,\quad \forall f\in C^\infty(B),
\end{align}
where $f_B = V(B)^{-1} \int_B f d\mu$.

We start by introducing some results about the Riesz transform. By \cite{CD2}, we know that the conjunction of \eqref{Doubling} and \eqref{DUE} implies \eqref{R_p} for all $1<p\le 2$, where \eqref{DUE} refers to the on-diagonal upper bound for heat kernel:
\begin{equation}\tag{$\textrm{DUE}$}\label{DUE}
    e^{-t\Delta}(x,x) \le \frac{C}{V(x,\sqrt{t})}, \quad \forall x\in M.
\end{equation}
In fact, it follows from \cite{grigor_heatkernel_Gaussian} that under the assumption of \eqref{Doubling}, the condition \eqref{DUE} above self-improves to
\begin{equation}\tag{$\textrm{UE}$}\label{UE}
    e^{-t\Delta}(x,y) \le \frac{C}{V(x,\sqrt{t})} e^{-\frac{d(x,y)^2}{ct}},\quad \forall x,y\in M. 
\end{equation}
While for $p>2$, a result from \cite{ACDH} indicates that under the assumptions of \eqref{Doubling} and $(\textrm{P}_2)$, \eqref{R_p} is equivalent to the $L^p\textit{-}$boundedness of the gradient of the heat semigroup. Recently, Jiang, see \cite{Jiang}, generalized the above criterion, showing that \eqref{R_p} is also equivalent to the reverse Hölder inequality for the gradient of harmonic functions. In \cite{CCH}, the authors consider the connected sum $\mathbb{R}^n \# \mathbb{R}^n$ for $n \ge 3$ (see \cite{GS} for a more detailed discussion of connected sums) and prove that \eqref{R_p} on such manifold only holds for $1<p<n$. Subsequently, the authors of \cite{HS} generalize the above result by considering a class of non-doubling manifolds, which fail to meet both \eqref{Doubling} and $(\textrm{P}_2)$, see also \cite{HNS}. 

The boundedness of the Riesz transform can also be examined from the view of Ricci curvature. Let $o\in M$ be a fixed point. We say $M$ has \textit{quadratically decaying curvature} if its Ricci curvature satisfies property (see \cite{LS} for details):
\begin{align}\tag{QD}\label{QD}
    \textrm{Ric}_x \ge - \frac{\delta^2}{\left(1+r(x)\right)^2} g_x,\quad \forall x\in M,
\end{align}
where $\delta \in \mathbb{R}$, and $g$ stands for the Riemannian metric. Here, we use notation $r(x) = d(x,o)$. Manifolds with conical ends and the connected sums of several copies of $\mathbb{R}^n$ are two typical examples which satisfy this \eqref{QD} condition; see \cite{CCH,GH1,GH2,L,LS}. By \cite{B}, it is well-known that \eqref{R_p} holds for all $1<p<\infty$ on manifolds with non-negative Ricci curvature. 

To go further, let us further introduce two assumptions. The first one is the volume comparison condition. That is, for any $R\ge 1$ and any $x\in \partial B(o,R)$, we have
\begin{equation}\tag{VC}\label{VC}
    V(o,R) \le C V(x,R/2).
\end{equation}
The second assumption is the "relatively connected to an end" condition (we may denote it by $(\textrm{RCE})$ through this article). Suppose $M$ has at most finitely many ends. Then we say $M$ satisfies $(\textrm{RCE})$ if there exists a constant $\theta\in (0,1)$ such that for all $R\ge 1$ and all $x\in \partial B(o,R)$, there is a continuous path $\gamma:[0,1]\to B(o,R) \setminus B(o,\theta R)$ and a geodesic ray $\tau:[0,\infty)\to M\setminus B(o,R)$ such that
\begin{align*}
    \bullet \gamma(0)=x, \gamma(1)=\tau(0),\quad \bullet \textrm{length}(\gamma)\le R/\theta.
\end{align*}
This (RCE) condition generalizes the condition (RCA) introduced in \cite{GS_2005} to the setting of manifolds with finitely many ends. For more details about $(\textrm{RCE})$, we refer readers to \cite{GS_2005,Gilles}. Next, \cite[Theorem~2.4]{Gilles} tells us $\eqref{VC} + \eqref{QD}+(\textrm{RCE})$ implies \eqref{Doubling} and \eqref{DUE}.

By \cite[Theorem~A]{Gilles}, under the assumptions of \eqref{QD}, \eqref{VC}, $(\textrm{RCE})$ and \eqref{RD} $(\nu>2)$ for 'anchored balls', the Riesz transform is bounded on $L^p$ for $1<p<\nu$. Note that this result is sharp in the sense of $\mathbb{R}^n \# \mathbb{R}^n$. 

As noted above, on certain classes of manifolds, inequality \eqref{R_p} only holds over a bounded range of exponents, in contrast to the Euclidean setting. This naturally prompts the question of whether a suitable endpoint estimate for $\nabla \Delta^{-1/2}$ exists. 

Before introducing the endpoint results, we have to recall some notations from Lorentz spaces. We say $f\in L^{p,q}$ $(0<p,q\le \infty)$ if the quasi-norm
\begin{equation*} 
     \|f\|_{(p,q)}=  
    \begin{cases}
    \left( \int_0^{\infty} \left(t^{1/p} f^*(t)\right)^q \frac{dt}{t}   \right)^{1/q}, & q<\infty,\\
     \sup_{t>0} t^{1/p}f^*(t) = \sup_{\lambda>0} \lambda d_f(\lambda)^{1/p}, & q=\infty.
    \end{cases}
\end{equation*}
is finite, where $f^*$ denotes the decreasing rearrangement function: $f^*(t) = \inf\{\lambda>0; d_f(\lambda)\le t\}$, and $d_f$ is the usual distribution function, i.e., $d_f(\lambda) = \mu(\{x; |f(x)| >\lambda  \})$.

In \cite{H}, building on the framework introduced in \cite{HS}, the author establishes a Lorentz-type endpoint estimate. Specifically, it is shown that $\nabla \Delta^{-1/2}$ is bounded on $L^{p_0,1}$, yet it is not bounded from $L^{p_0,p}$ to $L^{p_0,q}$ for any $1<p<\infty$ and $p\le q\le \infty$, where $p_0$ is the endpoint exponent for the boundedness of the Riesz transform. This finding aligns with the result in \cite{L}, which demonstrates that on a class of conic manifolds, $\nabla \Delta^{-1/2}$ fails to be of weak type bounded at the endpoint.

Before presenting our results, let us briefly discuss the motivation behind these endpoint estimates. In \cite{CCH}, the authors study the connected sums $M=\mathbb{R}^n \# \mathbb{R}^n$ and show that on a compactification of $M$, which we denote by $\overline{M}$, the Riesz potential exhibits an asymptotic expansion 
\begin{align}\label{asymptotics}
    \Delta^{-1/2}(x,y) \sim \sum_{j=n-1}^\infty a_j(x) |y|^{-j},\quad x\in \overline{M},\quad |y|\to \partial \overline{M},
\end{align} 
where $a_{n-1}$ is a nontrivial bounded harmonic function on $\overline{M}$. In particular, when $M=\mathbb{R}^n$, a similar asymptotic behavior holds with $a_{n-1}=1$.

Applying a gradient operator in the $x\textit{-}$variable, the maximum principle ensures that $\nabla a_{n-1}$ does not vanish, resulting in a leading term in the $y\textit{-}$variable that decays like $|y|^{1-n}$ near the boundary. Note that the function $|y|^{1-n}$ can only pair boundedly with functions in $L^p$ for $p<n$, which in a sense explains why the boundedness range of the Riesz transform terminates at $p=n$. By contrast, when $M=\mathbb{R}^n$, the gradient annihilates the leading term because $a_{n-1}=1$. Consequently, the kernel then decays like $|y|^{-n}$ near the boundary, making it can be boundedly pair with functions in $L^p$ for all $1<p<\infty$. 

Next, to address the endpoint issue, note that for $M=\mathbb{R}^n \# \mathbb{R}^n$, a straightforward calculation shows that $|y|^{1-n}$ lies in the Lorentz space $L^{\frac{n}{n-1},\infty}(\partial \overline{M})$. By duality, this implies that $|y|^{1-n}$ should be able to pair with functions in $L^{n,1}$. Although the asymptotic expansion \eqref{asymptotics} has only been verified on $\mathbb{R}^n \# \mathbb{R}^n$, one may naturally conjecture that it holds for a broader class of manifolds. 

Our first result supplements \cite[Theorem~A]{Gilles} and can be stated as follows.

\begin{theorem}\label{R_Endpoint}
Let $M$ be a complete Riemannian manifold satisfying \eqref{QD}, \eqref{VC}, $(\textrm{RCE})$ and reverse doubling condition for "anchored balls" i.e.,
\begin{align}\tag{$\textrm{RD}_\nu^{o}$}\label{RD0}
    C\left(\frac{R}{r}\right)^\nu \le  \frac{V(o,R)}{V(o,r)},\quad \forall R\ge r>0.
\end{align}
for some $\nu>2$. Then the Riesz transform, $\nabla \Delta^{-1/2}$, is of restricted weak type $(\nu,\nu)$. That is,
\begin{equation*}
    \|\nabla \Delta^{-1/2}f\|_{(\nu,\infty)} \le C\|f\|_{(\nu,1)}
\end{equation*}
for all $f\in C_c^\infty(M)$.
\end{theorem}

We note that the condition $\nu>2$ is not essential; see also \cite[Remark~3.1]{Gilles}, the assumptions of Theorem~\ref{R_Endpoint} already imply \eqref{R_p} for all $1<p\le 2$, independent of $\nu$. Consequently, a Lorentz-type endpoint result at $p=\nu$ for $1<\nu\le 2$  carries no additional significance; see also Remark~\ref{remark_weak11} below.

\medskip

In our next investigation, we study the reverse Riesz inequality. By \cite[Proposition~2.1]{CD}, it is well-known that  on any complete Riemannian manifolds, \eqref{R_p} implies $(\textrm{RR}_{p'})$. However, the converse is not always true. Therefore, the following two natural questions arise: how to prove such a reverse inequality, and why is \eqref{eq_RRp} less demanding (compare to prove \eqref{R_p})? For the first question, a natural way is to use the duality argument directly (\cite[Proposition~2.1]{CD}). For example, by \cite{CD2}, \eqref{Doubling} and \eqref{DUE} implies \eqref{R_p} for $1<p\le 2$ and hence by duality, \eqref{eq_RRp} holds for $p\ge 2$, 

To address the above reverse problem more thoroughly, we present the following three illustrative examples. First, by establishing a Calderón–Zygmund decomposition for Sobolev functions, Auscher-Coulhon \cite{AC} showed that the conjunction of \eqref{Doubling} and \eqref{P_q} (for some $1\le q<2$) implies \eqref{eq_RRp} for all $q<p<\infty$. Under the same assumptions, however, one can only derive \eqref{R_p} for $1<p<2+\epsilon$ (for some $\epsilon>0$). Second, in \cite{KVZ}, one considers the Dirichlet Laplacian outside some convex bounded obstacle in $\mathbb{R}^d$ $(d\ge 3)$. By establishing a Littlewood-Paley theory, their results suggest that \eqref{eq_RRp} is valid for all $1<p<\infty$, whereas \eqref{R_p} only holds for $1<p<d$; see \cite{JL} for an analogous result when $d=2$. Third, in a recent article \cite{H2}, the author examines a class of manifolds that satisfy neither \eqref{Doubling} nor \eqref{P_q} (for $1\le q\le 2$), and for which the Riesz transform is only $L^p\textit{-}$bounded on a finite interval. Nonetheless, by using a so-called harmonic annihilation method, the results there still confirm that \eqref{eq_RRp} holds for every $1<p<\infty$. Further details can be found in \cite{AC,KVZ,JL,HS2,H2}.

To describe our next result, one recalls that for $1\le p<\infty$, the $L^p$ Hardy's inequality holds on $M$ if
\begin{align}\tag{$\textrm{H}_p$}\label{eq_Hp}
    \left\| \frac{f}{r(\cdot)} \right\|_{p} \le C\|\nabla f\|_p,\quad \forall f\in C_c^\infty(M).
\end{align}
Another ingredient of the statement is the so-called weighted Sobolev inequality. We say $M$ satisfies weighted Sobolev inequality if
\begin{align}\tag{$\textrm{WS}_p^\mu$}\label{eq_WSp}
\left\| \frac{f}{\rho(\cdot)} \right\|_{(p^*,p)} \le C \|\nabla f\|_p,\quad \forall f\in C_c^\infty(M), \quad p^* = \frac{\mu p}{\mu-p}, 
\end{align}
where $1\le p<\mu$ and
\begin{align*}
    \rho(x) = \frac{r(x)}{V(o,r(x))^{1/\mu}}.
\end{align*}
Note that here we ask for a Lorentz quasi-norm $L^{p^*,p}$ on the left of \eqref{eq_WSp}, which is the sharpest form of Sobolev inequality in the sense of $\mathbb{R}^n$; see \cite{BCLS,O'neil}.

We present our next result in the following way.

\begin{theorem}\label{RR_Equivalence}
Let $M$ be a complete Riemannian manifold satisfying \eqref{QD}, \eqref{VC}, \eqref{RD} $(\nu>1)$, and $(RCE)$. Then the following statements are equivalent:
\begin{enumerate}[label=(\roman*)]
    \item \eqref{eq_Hp} holds for $p\in (1,\nu)$, 
    \item \eqref{eq_RRp} holds for $p\in (1,\nu)$,
    \item \eqref{eq_WSp} holds for $p\in (1,\nu)$,
\end{enumerate}
where $\mu\ge \nu$ is the doubling exponent in \eqref{Doubling}.
\end{theorem}

In this article, rather than adapting the approach from \cite{AC} (see also \cite{DR}) based on conditions \eqref{Doubling} and \eqref{P_q}, we focus on the intrinsic connection between \eqref{eq_RRp} and the Riesz transform itself. Specifically, we extend the harmonic annihilation method introduced in \cite{H2} to manifolds with quadratically decaying curvature.

To elaborate our method; see also \cite{H2}, let us again consider $M=\mathbb{R}^n \# \mathbb{R}^n$. Let $f,g\in C_c^\infty(\overline{M})$ such that $g$ is supported near the boundary $\partial \overline{M}$. We consider bilinear form $\langle \Delta^{1/2}f, g \rangle$. By using a suitable resolution to identity, we have by \eqref{asymptotics},
\begin{align}\label{asymptotic2}
    \left\langle \Delta^{1/2}f, g \right\rangle = \left\langle \nabla f, \nabla \Delta^{-1/2}g \right\rangle \sim \left\langle \nabla f, \int_{\partial \overline{M}} \sum_{j=n-1}^\infty \nabla a_j(x) |y|^{-j} g(y) dy  \right\rangle.
\end{align}
While by integration by parts, the leading coefficent $\nabla a_{n-1}$ vanishes since $a_{n-1}$ is harmonic, which implies
\begin{align*}
    \left\langle \Delta^{1/2}f, g \right\rangle \sim \left\langle f, \int_{\partial \overline{M}} \sum_{j=n}^\infty \Delta a_j(x) |y|^{-j} g(y) dy  \right\rangle.
\end{align*}
Let $Q(x)$ be some potential. We rewrite the above bilinear form as
\begin{align*}
    \left\langle \Delta^{1/2}f, g \right\rangle \sim \left\langle \frac{f(x)}{Q(x)}, \int_{\partial \overline{M}} \sum_{j=n}^\infty Q(x) \Delta a_j(x) |y|^{-j} g(y) dy  \right\rangle.    
\end{align*}
Finally, we reduce the problem to the estimates of the left and right entries of the above bilinear form separately. In fact, in \cite{H2}, the author computed the leading coefficient $a_{n-1}$ explicitly. However, in the general setting of \eqref{QD} manifolds, we employ this harmonic annihilation method implicitly, see also \cite[Section~5, 6]{H2}.

By restricting our attention to manifolds with quadratically decaying curvature, our method explaines why \eqref{eq_RRp} is comparatively easier to satisfy than \eqref{R_p}: in the above bilinear setting, the problematic leading coefficient $\nabla a_{n-1}$ disappears.

Next, under additional geometric assumptions on the manifold, we essentially recover the result of \cite[Corollary~1.4]{DR}. 

We recall that $M$ is said to have bounded geometry if and only if its Ricci curvature is bounded from below and its radius of injectivity is strictly positive. By modifying the proof of Theorem~\ref{RR_Equivalence}, we obtain the following result.

\begin{theorem}\label{thm_RR_QD}
Let $M$ be a complete Riemannian manifold with bounded geometry satisfying \eqref{QD}, \eqref{VC}, \eqref{RD} $(\nu>1)$ and $(\textrm{RCE})$. Then \eqref{eq_RRp} holds on $M$ for all $p\in (1,\nu) \cup [2, \infty)$.
\end{theorem}

Note that for $n\ge 2$, the connected sum $\mathbb{R}^n \# \mathbb{R}^n$ has bounded geometry and satisfies all the assumptions in Theorem~\ref{thm_RR_QD}. We therefore regain the main results of \cite{H2} (a special case) in a different way. We also mention that under a similar setting, our result aligns with \cite[Corollary~1.7]{DR}. That is, \eqref{eq_RRp} holds for all $p\in (1,\infty)$ when $\nu>2$. The proof from \cite{DR} relied on a Calderón–Zygmund decomposition for Sobolev functions introduced in \cite{AC}. Here, however, we offer an alternative proof that leverages the intrinsic connection between the Riesz and reverse Riesz transforms.
\bigskip

Let $n=m=1$ and $\beta \ge 0$. From \cite{Grushin}, the original Grushin operator is given by $-\partial_x^2 - x^{2\beta} \partial_y^2$. For $n\ge 1$, $m\ge 1$ and $\beta \ge 0$. It is natural to define the generalized Grushin operator on $\mathbb{R}^{n+m}$ in the way:
\begin{align}\label{eq_grushin_operator}
    L = - \sum_{i=1}^n \partial_{x_i}^2 - |x|^{2\beta} \sum_{i=1}^m \partial_{y_i}^2 = \Delta_x + |x|^{2\beta} \Delta_y,\quad (x,y)\in \mathbb{R}^{n+m}.
\end{align}
Precisely, let $\nabla_L$ denote the gradient operator defined by $\nabla_L = \left(\nabla_x, |x|^\beta \nabla_y \right)$. We then define $L$ as the unique positive self-adjoint operator associated with the Friedrichs extension of the following Dirichlet form:
\begin{align*}
    Q(f) = \int_{\mathbb{R}^{n+m}} \nabla_Lf(\xi) \cdot \nabla_L f(\xi) d\xi
\end{align*}
for any $f\in C_c^\infty(\mathbb{R}^{n+m})$.

Denote by $\mathcal{R}$, the Riesz transform associated to $L$, i.e. $\mathcal{R}=\nabla_L L^{-1/2}$. From \cite[Theorem~8.1]{DS1} (see also \cite{CD2}), $\mathcal{R}$ is bounded on $L^p$ for all $1<p\le 2$. Moreover, it is of weak type $(1,1)$. By \cite{DS3}, if $\beta \in \mathbb{N}$, then \eqref{R_p} holds for all $p\in (1,\infty)$ by using techniques from nilpotent Lie group theory. It is also worth mentioning that for the special case where $m=\beta=1$, a similar result was obtained in \cite{JST} by using a quite different approach. 

In this note, we focus on the Riesz and reverse Riesz inequalities under the general setting, where $n\ge 2$, $m\ge 1$ and $\beta \ge 0$. For the reverse inequality, by duality, the result from \cite[Theorem~8.1]{DS1} indicates \eqref{eq_RRp} holds for all $p\ge 2$. Next, from \cite{DS2}, the Poincaré inequality $(\textrm{P}_2)$ holds on the Grushin space. Then, followed by \cite{SZ} and \cite{AC}, \eqref{eq_RRp} is expected to be hold for $p\in (2-\epsilon,\infty)$ for some $\epsilon >0$.

Consider metric
\begin{align}\label{eq_metric}
    g_\xi = dx^2 + |x|^{-2\beta} dy^2, \quad \xi = (x,y)\in \mathbb{R}^n \setminus \{0\} \times \mathbb{R}^m.
\end{align}
By considering the space as a doubly warped product space, one can easily verify that the space $(\mathbb{R}^{n}\setminus \{0\} \times \mathbb{R}^m, g_\xi)$ has Ricci lower bound: $\textrm{Ric}_\xi \ge -c(n,m,\beta)/|x|^2$. Note that for $\xi = (x,y)\in \mathbb{R}^n \setminus \{0\} \times \mathbb{R}^m$, $L$ coincides with the weighted Laplacian $\Delta_g + V$, where $\Delta_g$ is the Laplace-Beltrami operator according to the metric \eqref{eq_metric}, and $V$ is a first order drift term. Instead of estimating Ricci curvature, one can verify that the curvature dimension inequality $\textrm{CD}(-c_1/|x|^2, c_2)$ holds on $\mathbb{R}^n \setminus \{0\} \times \mathbb{R}^m$.

By adapting the methods from \cite{HL,Gilles,ACDH}, we analyze the good part, the diagonal part and the bad part of the Riesz kernel respectively. Combining this with the harmonic annihilation method, we verify the following result.

\begin{theorem}\label{thm_RR_Grushin}
Let $n\ge 2$, $m\ge 1$ and $\beta > 0$. Let $L$ be the Grushin operator given by \eqref{eq_grushin_operator}. Then the reverse Riesz inequality holds in the sense:
\begin{align*}
    \|L^{1/2} f\|_p \le C \|\nabla_L f\|_p, \quad \forall f\in C_c^\infty(\mathbb{R}^{n+m}) 
\end{align*}
for all $1<p<\infty$.

Moreover, the Riesz transform is bounded on $L^p$ in the sense:
\begin{align*}
    \|\nabla_L L^{-1/2}f\|_p \le C \|f\|_p,\quad \forall f\in C_c^\infty(\mathbb{R}^{n}\setminus \{0\} \times \mathbb{R}^m)
\end{align*}
for $1<p\le 2$ if $n=2$ and $1<p<n$ if $n>2$.
\end{theorem}

\section{Preliminaries}\label{section_2}
Throughout the paper, we use notations $A\lesssim B$, $A\gtrsim B$ and $A\simeq B$ to denote $A\le cB$, $A\ge cB$ and $cB \le A \le c^{-1} B$ respectively for some constant $c>0$. We also write $dx = d\mu(x) = d\textrm{vol}(x)$ for simplicity.
\medskip

One of \cite{Gilles}'s main results, which is the one we continue to study here, can be described in the following way. 

\begin{theorem}\cite[Theorem A]{Gilles}\label{A}
Assume $(M,g)$ is a complete Riemannian manifold satisfying \eqref{QD}, \eqref{VC} and $(RCE)$. If furthermore, $M$ satisfies \eqref{RD0} for some $\nu>2$, then the Riesz transform, $\nabla \Delta^{-1/2}$, is bounded on $L^p$ for all $p\in (1,\nu)$. 
\end{theorem}    

\begin{remark}\label{remark_key}
Here, as also remarked in \cite[Remark~3.1]{Gilles}, we note that the aforementioned result also holds for $1 < \nu \leq 2$. However, since the assumptions already imply \eqref{R_p} for all $1<p\le 2$ (see Remark~\ref{remark_weak11} later), the case where $1<\nu\le 2$ is not that interesting. In particular, we will use this observation in our argument later.
\end{remark}

We say $M$ satisfies the relative Faber-Krahn inequality if there exists $\alpha>0$ such that for all $x\in M$, $R>0$ and any open subset $\Omega \subset B(x,R)$, we have
\begin{equation}\tag{FK}\label{FK}
\lambda_1(\Omega) \gtrsim  R^{-2} \left(\frac{\mu(\Omega)}{V(x,R)}\right)^{-2/\alpha},   
\end{equation}
where $\lambda_1(\Omega)$ stands for the smallest eigenvalue of the Laplacian on $\Omega$ for the Dirichlet boundary condition.

\begin{remark}\label{remark_weak11}
Although not explicitly stated in Theorem~\ref{A}, $\nabla \Delta^{-1/2}$ is also of weak type $(1,1)$. Indeed, by combining results from \cite{Gilles}, \cite{CD2}, and \cite{Gri_heatkernel}, one obtains the following chain of implications: the assumptions of Theorem~\ref{A} imply \eqref{FK}, and, by \cite{Gri_heatkernel}, \eqref{FK} is equivalent to the conjunction of \eqref{Doubling} and \eqref{DUE}. Finally, as established in \cite{CD2}, \eqref{Doubling} and \eqref{DUE} ensure that $\nabla \Delta^{-1/2}$ is bounded on $L^p$ for $p\in (1,2]$ and of weak type $(1,1)$.  
\end{remark}

The main idea of \cite{Gilles} is to decompose the kernel of $\nabla \Delta^{-1/2}$ into three parts with respect to the following regimes: for some $\kappa \ge 4$,

$(1)$ $\left\{ (x,y)\in M\times M \setminus \{x=y\}; d(x,y)\ge \kappa^{-1} r(x), r(y) \le \kappa r(x) \right\}$.

$(2)$ $\left\{ (x,y)\in M\times M \setminus \{x=y\}; d(x,y) \le \kappa^{-1} r(x)\right\}$.

$(3)$ $\left\{ (x,y)\in M\times M \setminus \{x=y\}; r(y)\ge \kappa r(x)\right\}$.

For simplicity, we write $\nabla \Delta^{-1/2} = T_1+T_2+T_3$, where $T_j(x,y)$ is the kernel of Riesz transform restricted to the regime $(j)$, $(j=1,2,3)$. The following lemma is a combination of \cite[Proposition 4.1, 4.2; Lemma 3.2]{Gilles} (also see Remark~\ref{remark_key}).

\begin{lemma}\cite{Gilles}\label{Lemma_A}
Let $M$ be a complete Riemannian manifold satisfying \eqref{QD}, \eqref{VC}, $(RCE)$ and \eqref{RD0} for some $\nu>1$. We have the following results:
\begin{enumerate}[label=(\roman*)]
    \item $T_1$ is bounded on $L^p$ for all $p\in (1,\infty)$.
    \item $T_2$ is bounded on $L^p$ for all $p\in (2,\infty)$.
    \item $T_3$ has kernel estimate: for any $x,y \in M$,
\begin{equation*}
    |T_3(x,y)| \lesssim \frac{r(y)}{r(x)} \frac{1}{V(o,r(y))}.
\end{equation*}
\end{enumerate}

\end{lemma}

\medskip

Next, for later use (the proof of Theorem~\ref{thm_RR_QD}), one considers $M$ satisfying \eqref{QD} and bounded geometry condition. We keep the notations from \cite{DR} for convenience.

\begin{definition}
Let $o\in M$ and $r_0>0$. We say that the ball $B=B(x,r)$
\begin{enumerate}[label=(\roman*)]
    \item is remote if $r\le r(x)/2$,
    \item is anchored if $x=o$,
    \item is admissible if $B$ is remote or anchored with radius $r\le r_0$.
\end{enumerate}

\end{definition}

\begin{definition}
$M$ is said to satisfy $L^q\textit{-}$Poincaré inequality on the end if for every admissible ball $B$,
\begin{equation}\tag{$\textrm{PE}_q$}\label{$P_q'$}
\|f-f_B\|_{L^q(B)} \lesssim r \|\nabla f\|_{L^q(B)}, \quad \forall f\in C^\infty(B),
\end{equation}
where $r$ is the radius of $B$.
\end{definition}

The following lemma is a direct consequence of \cite[Theorem~2.4]{Gilles} and \cite[Theorem~2.3]{DR}.

\begin{lemma}\label{lemma_513}
Let $M$ be a complete Riemannian manifold with bounded geometry which satisfies \eqref{QD}, \eqref{RD} $(\nu>1)$, \eqref{VC} and $(\textrm{RCE})$. Then one deduces
\begin{enumerate}[label=(\roman*)]
    \item \eqref{Doubling} and \eqref{DUE} hold on $M$.
    \item The following version $L^p\textit{-}$Hardy inequality holds 
\begin{equation*}
    \int_M \left( \frac{|f(x)|}{1+r(x)}  \right)^p dx \lesssim \int_M |\nabla f|^p dx, \quad \forall f\in C_c^\infty(M)
\end{equation*}
for all $1\le p<\nu$.
\end{enumerate}
\end{lemma}

\begin{proof}[Proof of Lemma~\ref{lemma_513}]
For $(\romannumeral1)$, see \cite[Theorem 2.4]{Gilles} or Remark~\ref{remark_weak11}. Next, for $(\romannumeral2)$, \cite[Theorem 5.6.5]{Saloff_aspect} guarantees \eqref{$P_q'$} holds on $M$ for all $q\ge 1$ (since $M$ satisfies \eqref{QD}). Therefore, it follows by \cite[Theorem 2.3]{DR}, see also \cite{Minerbe}, that \eqref{Doubling},$(\textrm{RCE})$,\eqref{$P_q'$} and \eqref{RD} imply the above $L^p$ Hardy's inequality for all $1\le p < \nu$.

\end{proof}

\section{Endpoint Estimate for Riesz Transform}\label{section_3}

In this section, we give a proof for Theorem~\ref{R_Endpoint}.

\begin{proof}[Proof of Theorem~\ref{R_Endpoint}]
Recall notations from Section~\ref{section_2}. We decompose the Riesz transform into three parts: 
\begin{equation*}
    \nabla \Delta^{-1/2} = T_1 + T_2 + T_3.
\end{equation*}
It follows by Lemma~\ref{Lemma_A} (also see \cite[Proposition 4.1, 4.2]{Gilles}) that $T_1,T_2$ are bounded on $L^\nu$ and hence are of restricted weak type $(\nu,\nu)$. Therefore, it suffices to study operator
\begin{equation*}
    \mathcal{T}: f \mapsto r(\cdot)^{-1} \int_{B(o,\kappa r(\cdot))^c} \frac{r(y)}{V(o,r(y))} |f(y)| dy.
\end{equation*}
Define function
\begin{equation*}
    \mathcal{G}_x(y) = \frac{r(y) }{V(o,r(y))} \chi_{M\setminus B(o,\kappa r(x))}(y).
\end{equation*}
It follows by Hardy-Littlewood inequality and Hölder's inequality,
\begin{align*}
    \mathcal{T}f(x) \le r(x)^{-1} \int_0^\infty f^*(t) \mathcal{G}_x^*(t) dt\\
    \le r(x)^{-1} \|f\|_{(\nu,1)} \|\mathcal{G}_x\|_{(\nu',\infty)},
\end{align*}
where 
\begin{align*}
    \|\mathcal{G}_x\|_{(\nu',\infty)}^{\nu'} &= \sup_{\lambda>0} \lambda^{\frac{\nu}{\nu-1}} \textrm{vol} \left( \left\{y\in M; \mathcal{G}_x(y) > \lambda  \right\} \right)\\
    &= \sup_{\lambda>0} \lambda^{\frac{\nu}{\nu-1}} \textrm{vol} \left( \left\{y\in M; r(y)\ge \kappa r(x) \quad \textrm{and} \quad V(o,r(y)) < \frac{r(y)}{\lambda}  \right\} \right)\\
    &:= \sup_{\lambda>0} \lambda^{\frac{\nu}{\nu-1}} \textrm{vol}(\mathcal{I}).
\end{align*}
Note that since $r(y)\ge \kappa r(x)$, one deduces by \eqref{RD0}
\begin{align*}
    \left(\frac{r(y)}{\kappa r(x)}\right)^\nu \lesssim \frac{V(o,r(y))}{V(o,\kappa r(x))}.
\end{align*}
Hence, for $y\in \mathcal{I}$,
\begin{align*}
    V(o,r(y)) < \frac{r(y)}{\lambda} = \frac{\kappa r(x)}{\lambda} \frac{r(y)}{\kappa r(x)} \le C \frac{\kappa r(x)}{\lambda} \frac{V(o,r(y))^{1/\nu}}{V(o,\kappa r(x))^{1/\nu}}.
\end{align*}
This implies
\begin{equation*}
    V(o,r(y)) \le C \lambda^{\frac{-\nu}{\nu-1}} \left(\frac{r(x)}{V(o,\kappa r(x))^{1/\nu}} \right)^{\frac{\nu}{\nu-1}}.
\end{equation*}
Consequently,
\begin{align*}
    \textrm{vol}(\mathcal{I}) &\le \textrm{vol} \left( \left\{ y\in M; V(o,r(y)) \le C \lambda^{\frac{-\nu}{\nu-1}} \left(\frac{r(x)}{V(o,\kappa r(x))^{1/\nu}} \right)^{\frac{\nu}{\nu-1}}  \right\} \right)\\
    &\le C \lambda^{\frac{-\nu}{\nu-1}} \left(\frac{r(x)}{V(o,\kappa r(x))^{1/\nu}} \right)^{\frac{\nu}{\nu-1}}.
\end{align*}
It is now plain that
\begin{align*}
    \|\mathcal{G}_x\|_{(\nu',\infty)} \lesssim \frac{r(x)}{V(o,\kappa r(x))^{1/\nu}}.
\end{align*}
To this end, one concludes that
\begin{align*}
    \mathcal{T}f(x) \lesssim \frac{\|f\|_{(\nu,1)}}{V(o,\kappa r(x))^{1/\nu}}.
\end{align*}
The proof is then complete since $V(o,\kappa r(x))^{-1/\nu} \in L^{\nu,\infty}$:
\begin{align*}
\sup_{\lambda >0} &\lambda^\nu \textrm{vol} \left( \left\{ x\in M; V(o,\kappa r(x))^{-1/\nu} > \lambda \right\} \right)\\
&=\sup_{\lambda >0} \lambda^\nu \textrm{vol} \left( \left\{ x\in M; V(o,\kappa r(x)) < \lambda^{-\nu} \right\} \right)\\
&\lesssim 1.
\end{align*}

\begin{remark}
We mention that this result aligns with the observation from \cite{L}: the Riesz transform on metric cone is not even of weak type $(p_0,p_0)$, where $p_0$ is an endpoint in the metric cone setting. Some similar results could be found in \cite{H} and the author's Ph.D. thesis \cite{He_phdthesis}. 
\end{remark}

\end{proof}

\section{From Hardy to Reverse Riesz}\label{section_4}

In this section, we verify the implication $(\romannumeral1)\implies (\romannumeral2)$ in Theorem~\ref{RR_Equivalence}.

As mentioned before, under the assumptions \eqref{QD}, \eqref{VC} and $(\textrm{RCE})$, $M$ satisfies \eqref{Doubling} and \eqref{DUE}. Then, a result of \cite{G1} guarantees the estimate for time-derivative of heat kernel:
\begin{equation*}
    |\partial_t e^{-t\Delta}(x,y)| \lesssim \frac{1}{t V(x,\sqrt{t})} e^{-\frac{d(x,y)^2}{ct}},\quad \forall x,y\in M,\quad \forall t>0.
\end{equation*}

Recall that $o\in M$ is a fixed point and we use notion $r(x) = d(x,o)$ for all $x\in M$.

\begin{theorem}\label{thm_H_RR}
Let $M$ be a complete Riemannian manifold satisfying \eqref{QD}, \eqref{VC}, \eqref{RD} $(\nu>1)$ and $(\textrm{RCE})$. Then, \eqref{eq_Hp} implies \eqref{eq_RRp} for all $p\in (1,\nu)$.
\end{theorem}

\begin{proof}[Proof of Theorem~\ref{thm_H_RR}]
First of all, by \cite[Theorem~A]{Gilles} and duality, we only need to confirm \eqref{eq_RRp} for
\begin{align}\label{range}
    p\in \begin{cases}
        (1,\nu), & 1<\nu\le 2,\\
        (1,\frac{\nu}{\nu-1}], & \nu>2.
    \end{cases}
\end{align}
Let $o\in M$ be fixed. By \cite{CheegerColding}, there exists a smooth function $\eta_\epsilon$ $(\epsilon>0)$ such that

$\bullet$ $\textrm{supp}(\eta_\epsilon) \subset B(o,4\epsilon/3)$,

$\bullet$ $\eta_\epsilon=1$ on $B(o,\epsilon)$,

$\bullet$ $\|\eta_\epsilon\|_\infty + \epsilon\|\nabla \eta_\epsilon\|_\infty + \epsilon^2 \|\Delta \eta_\epsilon\|_\infty \lesssim 1$.

Let $\kappa \ge 4$. Then $\eta_{\kappa^{-1} r(y)}(x)$ is a smooth function supported in the regime: $\{4r(y)\ge 3\kappa r(x)\}$. Next, for the Riesz transform, $\nabla \Delta^{-1/2}$, we have resolution to identity
\begin{align*}
    \nabla \Delta^{-1/2} = \int_0^\infty \nabla e^{-t\Delta} \frac{dt}{\sqrt{\pi t}}.
\end{align*}
Denote by $p_t(x,y)$ the heat kernel. We may decompose
\begin{align*}
    \nabla \Delta^{-1/2}(x,y) &= \int_0^\infty \eta_{\kappa^{-1}r(y)}(x) \nabla p_t(x,y) \frac{dt}{\sqrt{\pi t}} + \int_0^\infty \left[1 - \eta_{\kappa^{-1}r(y)}(x) \right] \nabla p_t(x,y) \frac{dt}{\sqrt{\pi t}}\\
    &:= R_1(x,y) + R_2(x,y).
\end{align*}
Note that $R_2$ is supported in the range $\{(x,y)\in M\times M \setminus \{x=y\}; r(y)\le \kappa r(x)\}$. Obviously, 
\begin{align*}
    |R_2(x,y)| \lesssim T_{1}(x,y) + T_{2}(x,y),
\end{align*}
where $T_{1}, T_{2}$ (introduced in Lemma~\ref{Lemma_A}) are operators of $\nabla \Delta^{-1/2}$ restricted to the domains:
\begin{align*}
    \left\{(x,y)\in M\times M \setminus \{x=y\}; r(y)\le \kappa r(x), d(x,y)\ge \kappa^{-1} r(x) \right\},
\end{align*}
and
\begin{align*}
    \left\{(x,y)\in M\times M \setminus \{x=y\}; d(x,y)\le \kappa^{-1}r(x)\right\}
\end{align*}
respectively. Then, it follows by Lemma~\ref{Lemma_A}, also see \cite[Proposition 4.1, 4.2]{Gilles}, that 
\begin{equation}\label{eq_R21,R22}
    \|R_2\|_{q\to q} \lesssim \|T_{1}+T_{2}\|_{q\to q} \lesssim 1
\end{equation}
for all $q\in (2,\infty)$.

Now, let $f,g\in C_c^\infty(M)$ with $\|g\|_{p'}=1$. We consider inner product
\begin{align*}
    \langle \Delta^{1/2}f, g \rangle = \left\langle \int_0^\infty \Delta e^{-t\Delta}f \frac{dt}{\sqrt{\pi t}}, g \right\rangle.
\end{align*}
By the positivity and self-adjointness of $\Delta$,
\begin{align*}
    \langle \Delta e^{-t\Delta}f, g \rangle = \langle f, \Delta e^{-t\Delta}g \rangle = \langle \nabla f, \nabla e^{-t\Delta}g \rangle, \quad \forall t>0.
\end{align*}
Hence,
\begin{align*}
    \langle \Delta^{1/2}f, g \rangle =  \left\langle \nabla f, \int_0^\infty \nabla e^{-t\Delta}g \frac{dt}{\sqrt{\pi t}} \right\rangle = \langle \nabla f, R_1 g \rangle + \langle \nabla f, R_2 g\rangle.
\end{align*}
By \eqref{eq_R21,R22}, it is plain that
\begin{align*}
    |\langle \nabla f, R_2 g\rangle|\le \|\nabla f\|_p \|R_2g\|_{p'} \lesssim \|\nabla f\|_p.
\end{align*}
Therefore, we may focus on the following bilinear form
\begin{equation*}
    \mathcal{I}(f,g) := \int_M \nabla f(x) \cdot \int_M \int_0^\infty \eta_{\kappa^{-1}r(y)}(x) \nabla p_t(x,y) g(y)\frac{dt}{\sqrt{\pi t}} dy dx.
\end{equation*}
Apply integration by parts to get
\begin{align*}
    \mathcal{I}(f,g) &= \int_M f(x) \int_M \int_0^\infty \eta_{\kappa^{-1}r(y)}(x) \Delta_x p_t(x,y) g(y) \frac{dt}{\sqrt{\pi t}} dy dx\\
    &- \int_M f(x) \int_M \int_0^\infty \nabla_x [\eta_{\kappa^{-1}r(y)}(x)] \cdot \nabla p_t(x,y) g(y) \frac{dt}{\sqrt{\pi t}} dy dx.
\end{align*}

The following estimate is crucial.

\begin{lemma}\label{lemma_akcakn}
Under the assumptions of Theorem~\ref{thm_H_RR}, one has
\begin{align*}
    \int_0^\infty |\nabla_x [\eta_{\kappa^{-1}r(y)}(x)] \cdot \nabla p_t(x,y)| + |\eta_{\kappa^{-1}r(y)}(x) \Delta_x p_t(x,y)| \frac{dt}{\sqrt{t}} \lesssim \frac{\chi_{r(x)\le 4\kappa^{-1}r(y)/3}}{r(y) V(o, r(y))}. 
\end{align*}
\end{lemma}

\begin{proof}[Proof of Lemma \ref{lemma_akcakn}]
Recall that on $M$ (see for example \cite[Section 3.2, 3.3]{Gilles}), we have gradient estimate
\begin{equation*}
    |\nabla p_t(x,y)| \le \left[ \frac{1}{\sqrt{t}} + \frac{1}{r(x)} \right] \frac{C}{V(x,\sqrt{t})} e^{-\frac{d(x,y)^2}{ct}}, \quad \forall x,y\in M, \quad \forall t>0.
\end{equation*}
Note that for $r(y)\ge \kappa r(x)$, $d(x,y)\simeq r(y)$. Moreover, for $x\in \textrm{supp}(\nabla \eta_{R})$ ($R>0$), we have $r(x)\simeq R$ (to be precise, in our situation, $\nabla \eta_{R}$ supports on $\{\kappa^{-1}r(y)\le r(x) \le 4\kappa^{-1}r(y)/3\}$). It is then plain that by \eqref{Doubling},
\begin{align*}
    \int_0^{r(x)^2} |\nabla_x [\eta_{\kappa^{-1}r(y)}(x)]| |\nabla p_t(x,y)| \frac{dt}{\sqrt{t}} &\lesssim \int_0^{r(x)^2} \frac{\chi_{r(x)\simeq r(y)}}{r(y)} \frac{e^{-\frac{d(x,y)^2}{ct}}}{V(x,\sqrt{t})} \frac{dt}{t}\\
    &\lesssim \frac{\chi_{r(x)\simeq r(y)}}{r(y) V(x,r(y))} \int_0^{r(y)^2} e^{-\frac{r(y)^2}{ct}} \left(\frac{r(y)}{\sqrt{t}}\right)^\mu \frac{dt}{t}\\
    &\lesssim \frac{\chi_{r(x)\le 4\kappa^{-1}r(y)/3}}{r(y) V(o, r(y))}.
\end{align*}
Similarly, one can use \eqref{RD} instead:
\begin{align*}
    \int_{r(x)^2}^\infty |\nabla_x [\eta_{\kappa^{-1}r(y)}(x)]| |\nabla p_t(x,y)| \frac{dt}{\sqrt{t}} &\lesssim \int_{r(x)^2}^\infty \frac{\chi_{r(x)\simeq r(y)}}{r(y) r(x)} \frac{e^{-\frac{d(x,y)^2}{ct}}}{V(x,\sqrt{t})} \frac{dt}{\sqrt{t}}\\
    &\lesssim \frac{\chi_{r(x)\simeq r(y)}}{r(y)^2 V(o,r(y))} \int_{r(y)^2}^\infty e^{-\frac{r(y)^2}{ct}} \left(\frac{r(y)}{\sqrt{t}}\right)^\nu \frac{dt}{\sqrt{t}}\\
    &\lesssim \frac{\chi_{r(x)\le 4\kappa^{-1}r(y)/3}}{r(y) V(o, r(y))}.
\end{align*}
As for the second part of the integral, recall time-derivative estimate:
\begin{equation*}
    |\partial_t p_t(x,y)| \le \frac{C}{t V(x,\sqrt{t})} e^{-\frac{d(x,y)^2}{ct}}, \quad \forall x,y\in M,\quad \forall t>0.
\end{equation*}
One deduces
\begin{align*}
    \int_0^\infty |\eta_{\kappa^{-1}r(y)}(x)| |\Delta_x p_t(x,y)| \frac{dt}{\sqrt{t}} &\lesssim  \int_0^\infty \frac{t^{-\frac{3}{2}} \chi_{r(x)\le 4\kappa^{-1}r(y)/3}}{V(x,\sqrt{t})} e^{-\frac{r(y)^2}{ct}} dt\\
    &\lesssim \int_0^{r(y)^2} \frac{t^{-\frac{3}{2}} \chi_{r(x)\le 4\kappa^{-1}r(y)/3}}{V(o,r(y))} \left(\frac{r(y)}{\sqrt{t}}\right)^\mu e^{-\frac{r(y)^2}{ct}} dt\\
    &+ \int_{r(y)^2}^\infty \frac{t^{-\frac{3}{2}} \chi_{r(x)\le 4\kappa^{-1}r(y)/3}}{V(o,r(y))} \left(\frac{r(y)}{\sqrt{t}}\right)^\nu e^{-\frac{r(y)^2}{ct}} dt\\
    &\lesssim \frac{\chi_{r(x)\le 4\kappa^{-1}r(y)/3}}{r(y) V(o, r(y))}
\end{align*}
as desired.
\end{proof}

Define operator
\begin{equation*}
    \mathcal{T}: u \mapsto r(\cdot) \int_{B(o,3\kappa r(\cdot)/4)^c} \frac{u(y)}{r(y) V(o, r(y))}dy.
\end{equation*}

The next ingredient of the proof is the following.

\begin{lemma}\label{lemma_njcndjcn}
Under the assumptions of Theorem~\ref{thm_H_RR}, $\mathcal{T}$ is of weak type $(q',q')$ for all $1<q<\infty$.
\end{lemma}

\begin{proof}[Proof of Lemma~\ref{lemma_njcndjcn}]
By Hölder's inequality, for all $1<q<\infty$,
\begin{align*}
    |\mathcal{T}u(x)|\lesssim r(x) \|u\|_{q'} \left( \int_{B(o,3\kappa r(\cdot)/4)^c} \frac{dy}{r(y)^q V(o, r(y))^q} \right)^{\frac{1}{q}}.
\end{align*}
Introduce the Riemann-Stieljes measure: $V(r) = V(o,r)$. One infers that by setting $\delta=3\kappa r(x)/4$ and integrating by parts, the integral on the RHS above equals
\begin{align*}
    \int_\delta^\infty \frac{1}{r^q V(r)^q} dV(r) = - \frac{1}{\delta^q V(\delta)^{q-1}}
    - \int_\delta^\infty V(r) \left[-q r^{-q-1} V(r)^{-q} dr - q V(r)^{-q-1} r^{-q} dV(r) \right].
\end{align*}
Hence, one concludes
\begin{align*}
    \int_\delta^\infty \frac{1}{r^q V(r)^q} dV(r) \lesssim \frac{1}{\delta^q V(\delta)^{q-1}} + \int_\delta^\infty \frac{1}{r^{q+1} V(r)^{q-1}} dr.
\end{align*}
By \eqref{RD}, 
\begin{equation*}
    \frac{1}{V(r)^{q-1}} \lesssim \frac{\delta^{\nu(q-1)}}{V(\delta)^{q-1}} r^{-\nu(q-1)}.
\end{equation*}
Consequently, for all $1<q<\infty$,
\begin{align*}
    \int_\delta^\infty \frac{1}{r^{q+1} V(r)^{q-1}} dr \lesssim \frac{\delta^{\nu(q-1)}}{V(\delta)^{q-1}} \int_\delta^\infty r^{-q-1-\nu(q-1)} dr \lesssim \delta^{-q} V(\delta)^{1-q}.
\end{align*}
Immediately, we deduce
\begin{align*}
    |\mathcal{T}u(x)| \lesssim \frac{\|u\|_{q'}}{V(o,r(x))^{\frac{q-1}{q}}}.
\end{align*}
As a consequence, for all $\lambda>0$,
\begin{align*}
    \textrm{vol} \left(\left\{ x\in M; |\mathcal{T}u(x)| > \lambda  \right\}\right) &\lesssim \textrm{vol} \left(\left\{ x\in M; V(o,r(x)) < \left(\frac{\|u\|_{q'}}{\lambda} \right)^{q'}  \right\}\right) \\
    &\lesssim \left(\frac{\|u\|_{q'}}{\lambda} \right)^{q'},
\end{align*}
i.e. $\mathcal{T}$ is of weak type $(q',q')$ for all $1<q<\infty$.
\end{proof}

Now, by interpolation, one gets in particular that for $p$ in the range \eqref{range},
\begin{equation*}
    \|\mathcal{T}g\|_{p'} \lesssim \|g\|_{p'}.
\end{equation*}
Therefore, by Hardy inequality \eqref{eq_Hp}, and Lemma~\ref{lemma_akcakn}, Lemma~\ref{lemma_njcndjcn}
\begin{align*}
    |\mathcal{I}(f,g)| &\le \left|\int_M f(x) \int_M \int_0^\infty \eta_{\kappa^{-1}r(y)}(x) \Delta_x p_t(x,y) g(y) \frac{dt}{\sqrt{\pi t}} dy dx\right|\\
    &+ \left|\int_M f(x) \int_M \int_0^\infty \nabla_x [\eta_{\kappa^{-1}r(y)}(x)] \cdot \nabla p_t(x,y) g(y) \frac{dt}{\sqrt{\pi t}} dy dx\right|\\
    &\lesssim \int_M \frac{|f(x)|}{r(x)} |\mathcal{T}g(x)| dx \lesssim \left\| \frac{f(\cdot)}{r(\cdot)} \right\|_p \|\mathcal{T}g\|_{p'}\\
    &\lesssim \|\nabla f\|_p \|g\|_{p'}.
\end{align*}
It then follows by duality
\begin{align*}
    \|\Delta^{1/2}f\|_p &\le \sup_{\|g\|_{p'}=1} |\langle \nabla f, R_1g \rangle| + \sup_{\|g\|_{p'}=1}|\langle \nabla f, R_2g \rangle| \\
    &\lesssim \sup_{\|g\|_{p'}=1} |\mathcal{I}(f,g)| + \sup_{\|g\|_{p'}=1} \|\nabla f\|_p \|R_2g\|_{p'}\\
    &\lesssim \|\nabla f\|_{p}.
\end{align*}
The proof of Theorem~\ref{thm_H_RR} is now complete.

\end{proof}

\begin{remark}
Note that the above proof can also be used to show the implication: $(\romannumeral3) \implies (\romannumeral2)$. Indeed, by the argument above, one only needs to show the inner product $\mathcal{I}(f,g)$ is bounded by some constant multiple of $\|\nabla f\|_p \|g\|_{p'}$. Now, instead of treating $\mathcal{T}$ in Lemma~\ref{lemma_njcndjcn}, we consider operator
\begin{align*}
    \Tilde{\mathcal{T}}: u \mapsto \frac{r(\cdot)}{V(o,r(\cdot))^{1/\mu}} \int_{B(o,3\kappa r(\cdot)/4)^c} \frac{u(y)}{r(y) V(o, r(y))}dy.
\end{align*}
By some similar estimates as in Lemma~\ref{lemma_njcndjcn}, one checks that 
\begin{align*}
    |\Tilde{\mathcal{T}}u(x)| \lesssim \frac{\|u\|_{q'}}{V(o,r(x))^{1/q' + 1/\mu}} = \frac{\|u\|_{q'}}{V(o,r(x))^{1/(q^*)'}},
\end{align*}
where $q^* = \frac{\mu q}{\mu - q}$ for all $1<q<\nu\le \mu$. Immediately, one concludes that $\Tilde{\mathcal{T}}$ is of weak type $(q', (q^*)')$. Hence, by interpolation, one deduces for all $0<s\le \infty$,
\begin{align*}
    \|\Tilde{\mathcal{T}}u\|_{\left((p^*)',s\right)} \lesssim \|u\|_{(p',s)},\quad \forall u\in C_c^\infty(M).
\end{align*}
Choosing $s=p'$, one confirms by Lemma~\ref{lemma_akcakn} and \eqref{eq_WSp} that
\begin{align*}
    |\mathcal{I}(f,g)|\lesssim \int_M \frac{f(x)}{\rho(x)} |\Tilde{\mathcal{T}}g(x)| dx \le \|f/\rho \|_{(p^*,p)} \|\Tilde{\mathcal{T}}g\|_{\left((p^*)',p'\right)} \lesssim \|\nabla f\|_p \|g\|_{p'}
\end{align*}
as desired.

\end{remark}

\begin{remark}
It is also worth mentioning that in the above proof, the \eqref{RD} condition can be weakened to \eqref{RD0}. Additionally, it is interesting to note that the harmonic leading coefficient, within the setting of manifolds satisfying the \eqref{QD} curvature condition, appears in exactly the same range as the case studied in \cite{H2}, i.e., the range where $r(y)\gg r(x)$. The proof of Theorem~\ref{thm_H_RR} closely follows the approach in \cite[Theorem~1.3]{H2}, where a so-called implicit harmonic annihilation is employed.
\end{remark}

\section{From Reverse Riesz to Weighted Sobolev}\label{section_5}
In this section, we study weighted Sobolev inequality. Classically, one studies Sobolev inequality in the form
\begin{align}\label{S_n^p}
    \left(\int_M |f|^{\frac{np}{n-p}} dx \right)^{\frac{n-p}{np}} \lesssim \left(\int_M |\nabla f|^p dx \right)^{\frac{1}{p}},\quad 1\le p<n,
\end{align}
where $n$ usually stands for the dimension of $M$. However, such inequality can be disturbed by very simple perturbation of the geometry. For instance, on the cylinder $\mathbb{R}^n \times \mathbb{S}^{m-1}$, \eqref{S_n^p} fails for any $1\le p< n+m-1$. This simple example in a sense suggests that, generally speaking, the topological dimension may not reflect the "real" geometry of a manifold.

In \cite{Minerbe}, the author proved that on any complete $n\textit{-}$dimensional Riemannian manifolds $M$ $(n\ge 3)$. If $M$ has non-negative Ricci curvature and satisfies \eqref{RD0} with some $\nu>2$, then it admits a weighted Sobolev inequality
\begin{align}\label{WS_p_n}
    \left(\int_M |f/\rho_o|^{\frac{n p}{n-p}}dx\right)^{1-p/n} \lesssim \int_M |\nabla f|^p dx, \quad \forall f\in C_c^\infty(M),
\end{align}
with $p=2$ and
\begin{align*}
    \rho_o(x) = \frac{r(x)}{V(o,r(x))^{1/n}}.
\end{align*}
Subsequently, it follows by \cite[Theorem~1.1]{Tewodrose}, a special case, that if one further assumes Poincaré inequality \eqref{P_q} and $(\textrm{D}_n)$ for some $1\le q < \min (\nu,n)$, then \eqref{WS_p_n} holds with the same weight $\rho_o$ for $q\le p<\min(\nu,n)$.  

On $\mathbb{R}^n$, it is well-known that by \cite{O'neil}, the Sobolev inequalities admit its sharpest form
\begin{align*}
    \|f\|_{\left(\frac{np}{n-p},p\right)} \lesssim \|\nabla f\|_p,\quad 1\le p<n.
\end{align*}
Then, Bakry, Coulhon, Ledoux and Saloff-Coste \cite{BCLS} proved that the classical Sobolev inequality, including its sharpest Lorentz form, can be obtained by a family of "weak type" Sobolev inequalities.

In what follows, we may prove the weighted Sobolev inequality in a different way. The idea is based on \cite[Lemma~5.4]{KVZ}, see also \cite{Mael}, where the authors proved Hardy type inequalities via the estimates of Riesz potential. In particular, we verify the implication $(\romannumeral2)\implies (\romannumeral3)$ in Theorem~\ref{RR_Equivalence}.

\begin{lemma}\label{thm_RR_WS}
Let $M$ be a complete Riemannian manifold satisfying \eqref{Doubling}, \eqref{RD}, \eqref{DUE}. Then, \eqref{eq_RRp} implies \eqref{eq_WSp} for all $1<p<\nu$. 

\end{lemma}

\begin{proof}[Proof of Lemma~\ref{thm_RR_WS}]
Since $\Delta^{1/2}C_c^\infty(M)$ is dense in $L^q(M)$ for $1\le q<\infty$, see \cite[Lemma~1]{Russ}, one only needs to show 
\begin{align*}
    \mathcal{T}: u \mapsto \rho(\cdot)^{-1} \Delta^{-1/2}u(\cdot)
\end{align*}
is bounded from $L^p \to L^{p^*,p}$ for all $1<p<\nu$ and $u\in C_c^\infty(M)$.

Using the resolution to identity, we have
\begin{align*}
    \Delta^{-1/2} = \int_0^\infty e^{-t\Delta} t^{-1/2} dt,
\end{align*}
where we omit the non-essential constant.

It is well-known that by \cite{grigor_heatkernel_Gaussian}, the conjunction of \eqref{Doubling} and \eqref{DUE} implies \eqref{UE}. Thus, one has kernel estimate
\begin{align*}
    \mathcal{T}(x,y)&\le \rho(x)^{-1} \int_0^\infty \frac{C}{V(x,\sqrt{t})} e^{-\frac{d(x,y)^2}{ct}} t^{-1/2} dt\\
    &\lesssim \rho(x)^{-1} \frac{1}{V(x,d(x,y))} \int_0^{d(x,y)^2} \frac{V(x,d(x,y))}{V(x,\sqrt{t})} e^{-\frac{d(x,y)^2}{ct}} t^{-1/2} dt\\
    &+ \rho(x)^{-1} \frac{1}{V(x,d(x,y))} \int_{d(x,y)^2}^\infty \frac{V(x,d(x,y))}{V(x,\sqrt{t})} e^{-\frac{d(x,y)^2}{ct}} t^{-1/2} dt.
\end{align*}
By \eqref{Doubling}, the first term above is thus bounded by
\begin{align*}
    \frac{d(x,y)^\mu}{\rho(x) V(x,d(x,y))} \int_0^{d(x,y)^2} e^{-\frac{d(x,y)^2}{ct}} t^{\frac{1-\mu}{2}-1} dt \lesssim \frac{d(x,y)}{\rho(x) V(x,d(x,y))}.
\end{align*}
While for the second term, one uses \eqref{RD} instead and bounds it above by
\begin{align*}
    \frac{d(x,y)^\nu}{\rho(x) V(x,d(x,y))} \int_{d(x,y)^2}^\infty e^{-\frac{d(x,y)^2}{ct}} t^{\frac{1-\nu}{2}-1} dt \lesssim \frac{d(x,y)}{\rho(x) V(x,d(x,y))},
\end{align*}
since $1<\nu/p<\nu$.

Define 
\begin{align*}
    \epsilon(x) = \frac{\|u\|_p^{p/\mu}}{V(x,r(x))^{1/\mu} \mathcal{M}(u)(x)^{p/\mu}},
\end{align*}
where $\mathcal{M}$ stands for the Hardy-Littlewood maximal function. We set
\begin{align*}
    \mathcal{O} = \left\{x\in M; \epsilon(x) \le 1\right\}.
\end{align*}
For $x\in \mathcal{O}$, we proceed in a standard way. 
\begin{align*}
    |\mathcal{T}u(x)| &\lesssim \frac{1}{\rho(x)} \int_{d(x,y)\le \epsilon(x) r(x)} \frac{d(x,y)}{V(x,d(x,y))} |u(y)|dy\\
    &+ \frac{1}{\rho(x)} \int_{d(x,y)\ge \epsilon(x) r(x)} \frac{d(x,y)}{V(x,d(x,y))} |u(y)|dy := \mathcal{I} + \mathcal{J}.
\end{align*}
Observe that by \eqref{Doubling}
\begin{align}\nonumber
    \mathcal{I} &= \rho(x)^{-1} \sum_{j=0}^\infty \int_{\epsilon(x) r(x)/2^{j+1}< d(x,y)\le \epsilon(x) r(x)/2^j} \frac{d(x,y)}{V(x,d(x,y))} |u(y)|dy\\ \nonumber
    &\lesssim \frac{V(x,r(x))^{1/\mu}}{r(x)} \sum_{j=0}^\infty \frac{\epsilon(x) r(x)}{2^{j}} \frac{1}{V(x, \epsilon(x) r(x)/2^{j+1})} \int_{B(x,\epsilon(x) r(x)/2^j)} |u(y)|dy\\
    &\lesssim \epsilon(x) V(x,r(x))^{1/\mu} \mathcal{M}(u)(x) = \mathcal{M}(u)(x)^{1-p/\mu} \|u\|_p^{p/\mu}.
\end{align}
As for $\mathcal{J}$, one deduces by Hölder's inequality
\begin{align*}
    \mathcal{J} \le \rho(x)^{-1} \|u\|_p \left(\int_{B(x,\epsilon(x) r(x))^c} \frac{d(x,y)^{p'}}{V(x,d(x,y))^{p'}} dy\right)^{1/p'}.
\end{align*}
Introduce Riemann-Stieljes measure: $V(r) = V(x,r)$. One infers by integrating by parts, that the above integral equals (set $R = \epsilon(x) r(x)$ for a moment)
\begin{align*}
    \int_{R}^\infty \frac{r^{p'}}{V(r)^{p'}} dV(r) = \left[\frac{r^{p'}}{V(r)^{p'-1}}\right]_{r=R}^\infty - \int_{R}^\infty V(r) \left( p'r^{p'-1} V(r)^{-p'} dr -p'V(r)^{-p'-1}r^{p'}dV(r) \right),
\end{align*}
which gives
\begin{align*}
    \int_{R}^\infty \frac{r^{p'}}{V(r)^{p'}} dV(r) \lesssim \left| \left[\frac{r^{p'}}{V(r)^{p'-1}}\right]_{r=R}^\infty \right| + \int_{R}^\infty  \frac{r^{p'-1}}{V(r)^{p'-1}} dr.
\end{align*}
Note again by \eqref{RD}, one infers
\begin{align*}
    \frac{r^{p'-1}}{V(r)^{p'-1}} \lesssim \frac{R^{\nu(p'-1)}}{V(R)^{p'-1}} r^{-\nu(p'-1)+p'-1}.
\end{align*}
Consequently, since $1<p<\nu$, one concludes
\begin{align*}
    \mathcal{J} \lesssim \|u\|_p \frac{V(x,r(x))^{1/\mu}}{r(x)} \frac{\epsilon(x) r(x)}{V(x, \epsilon(x) r(x))^{1/p}}.
\end{align*}
Now, recall that for $x\in \mathcal{O}$, one has $\epsilon(x) \le 1$. Then, \eqref{Doubling} guarantees that
\begin{align}\nonumber
    \mathcal{J} &\lesssim \epsilon(x) \|u\|_p V(x,r(x))^{1/\mu-1/p} \left(\frac{V(x,r(x))}{V(x,\epsilon(x) r(x))}\right)^{1/p}\\
    &\lesssim \epsilon(x)^{1-\mu/p} \frac{\|u\|_p}{V(x,r(x))^{1/p^*}} = \mathcal{M}(u)(x)^{1-p/\mu} \|u\|_p^{p/\mu}.
\end{align}
Next, one considers $x\in \mathcal{O}^c$. It is plain that by a similar argument above and the assumption $\epsilon(x) \ge 1$, we have
\begin{align*}
    |\mathcal{T}u(x)|&\lesssim  \frac{1}{\rho(x)} \int_{d(x,y)\le r(x)} \frac{d(x,y)}{V(x,d(x,y))} |u(y)|dy\\
    &+ \frac{1}{\rho(x)} \int_{d(x,y)\ge r(x)} \frac{d(x,y)}{V(x,d(x,y))} |u(y)|dy\\
    &\lesssim V(x,r(x))^{1/\mu} \mathcal{M}(u)(x) + \frac{\|u\|_p}{V(x,r(x))^{1/p^*}}\\
    &\lesssim \frac{\|u\|_p}{V(x,r(x))^{1/p^*}}\lesssim \frac{\|u\|_p}{V(o,r(x))^{1/p^*}}.
\end{align*}
To this end, for $\lambda >0$, it is clear that by maximal theorem
\begin{align*}
    \textrm{vol} \left(\left\{ x\in M; |\mathcal{T}u(x)|>\lambda
 \right\}\right) &\lesssim \textrm{vol} \left(\left\{ x\in \mathcal{O}; \mathcal{M}(u)(x)^{1-p/\mu} \|u\|_p^{p/\mu} >\lambda
 \right\}\right)\\
 &+ \textrm{vol} \left(\left\{ x\in M\setminus \mathcal{O}; \frac{\|u\|_p}{V(o,r(x))^{1/p^*}} >\lambda
 \right\}\right) \\
 &\le \frac{\|u\|_p^{p^* p/\mu} \left\|\mathcal{M}(u)^{1-p/\mu}\right\|_{p^*}^{p^*}}{\lambda^{p^*}}\\
 &+ \textrm{vol} \left(\left\{ x; V(o,r(x)) < \left(\frac{\|u\|_p}{\lambda}\right)^{p^*} 
 \right\}\right)\\
 &\lesssim \left(\frac{\|u\|_p}{\lambda}\right)^{p^*}.
\end{align*}
That is $\mathcal{T}$ is of weak type $(p,p^*)$ for all $1<p<\nu$. By interpolation on Lorentz spaces, see for example \cite[Theorem~1.4.19]{G}, we conclude that for all $0<s\le \infty$,
\begin{align*}
    \|\mathcal{T}u\|_{(p^*,s)}\lesssim \|u\|_{(p,s)}.
\end{align*}
The proof of Lemma~\ref{thm_RR_WS} follows by choosing $s=p$.

\end{proof}

\begin{remark}
It is clear to see that the above argument also proves the implication $(\romannumeral2)\implies (\romannumeral1)$, which generalizes the result of \cite[Theorem~1.5]{Mael} (this simple generalization was also mentioned in \cite[Question~2.4]{DR}). In fact, one only needs to modify the estimates of $\mathcal{I}$ and $\mathcal{J}$. Replacing $\rho(x) = r(x)$ and $\epsilon = 1$, a similar argument above gives
\begin{align*}
    \mathcal{I} \lesssim \mathcal{M}(u)(x) \quad \textrm{and}\quad \mathcal{J}\lesssim \frac{\|u\|_{p'}}{V(o,r(x))^{1/p'}}.
\end{align*}
The implication $(\romannumeral2)\implies (\romannumeral1)$ then follows directly by maximal theorem and interpolation. 
\end{remark}

\section{From Weighted Sobolev to Hardy}\label{section_6}
To end the proof of Theorem~\ref{RR_Equivalence}, we have to confirm the implication $(\romannumeral3) \implies (\romannumeral1)$. Arguably, on Euclidean space, the classical Sobolev inequality and Hardy inequality are almost "equivalent", see for example \cite[Proposition~1.4.1, Proposition~1.4.3]{BEL}.

\begin{lemma}\label{lemma_S_H}
Let $M$ be a complete Riemannian manifold. Then, \eqref{eq_WSp} implies \eqref{eq_Hp}, where $\mu>1$ and $1\le p<\mu$.
\end{lemma}

\begin{proof}[Proof of Lemma~\ref{lemma_S_H}]
It is clear that by Hölder's inequality in Lorentz space,
\begin{align*}
    \left(\int_M \left| \frac{f(x)}{\rho(x)} \right|^p \left| \frac{\rho(x)}{r(x)}\right|^p dx \right)^{1/p} \le \left\| \frac{f}{\rho} \right\|_{(p^*,p)} \left\| \frac{\rho}{r}\right\|_{(\mu,\infty)}
\end{align*}
since
\begin{align*}
    \frac{1}{p} = \frac{1}{p^*} + \frac{1}{\mu}, \quad \frac{1}{p} = \frac{1}{p} + \frac{1}{\infty}.
\end{align*}
Now, by \eqref{eq_WSp}, the first term is apparently bounded by $\|\nabla f\|_p$. While for the second term, one checks easily for all $\lambda >0$,
\begin{align*}
    \left\| \frac{\rho}{r}\right\|_{(\mu,\infty)}^\mu &= \sup_{\lambda>0}\lambda^{\mu}\textrm{vol} \left( \left\{x\in M; \rho(x)/r(x) > \lambda \right\} \right) \\
    &= \sup_{\lambda>0}\lambda^{\mu} \textrm{vol} \left( \left\{x\in M; V(o,r(x)) < \lambda^{-\mu} \right\} \right)\\
    &\le 1.
\end{align*}
The proof follows immediately.

\end{proof}

\section{Application: proof of Theorem~\ref{thm_RR_QD}}\label{section_7}
This section is devoted to prove Theorem~\ref{thm_RR_QD}.

Note that in \cite{DR}, the authors proved that under the assumptions of \eqref{Doubling}, \eqref{UE}, \eqref{RD} for some $\nu>1$ and $(\textrm{RCE})$. Then, Poincaré on the end \eqref{$P_q'$} $(1<q<\min(2,\nu))$ implies \eqref{eq_RRp} for $q\le p<2$. Next, as a corollary, \cite[Corollary 1.7]{DR}, for manifold with Ricci curvature bounded from below, satisfying \eqref{QD}, \eqref{VC}, $(\textrm{RCE})$ and \eqref{RD} $(\nu>1)$, \eqref{eq_RRp} holds for all $p\in (1,\infty)$. In the following, to better understand the interplay between Riesz and reverse Riesz transforms, we present an alternative proof by using the method established in previous sections. 

Most of the arguments are followed by the proof of Theorem~\ref{thm_H_RR}. We only need to make some necessary modifications. Recall notions $R_1, R_2,\mathcal{I},\eta,\mathcal{T}$ from Section~\ref{section_4}.

\begin{proof}[Proof of Theorem~\ref{thm_RR_QD}]

By duality, it suffices to consider \eqref{eq_RRp} with $p$ in the range \eqref{range}. The only obstacle is that the Hardy inequality we got (see Lemma~\ref{lemma_513}) is slightly different to the normal one \eqref{eq_Hp}.

Let $f\in C_c^\infty(M)$. Pick some cut-off function $\mathcal{X}\in C^\infty(M)$ such that $\textrm{supp}(\mathcal{X})\subset B(o,2)$ and $\mathcal{X}=1$ on $B(o,1)$. Split $f = f\mathcal{X} + f(1-\mathcal{X}):= f_0+f_1$. We start by claiming
\begin{align}\label{eq_claim_local}
    \|\Delta^{1/2}f_0\|_p \lesssim \|\nabla f\|_p
\end{align}
for all $p$ in the range \eqref{range}.

By spectral theory, one writes
\begin{align*}
    \Delta^{1/2} = \int_0^1 \Delta (\Delta+k^2)^{-1} dk + \int_1^\infty \Delta (\Delta+k^2)^{-1} dk := \mathcal{O}_L + \mathcal{O}_H.
\end{align*}

\begin{lemma}\label{lemma_RRp_high_energy}
Let $M$ be a complete Riemannian manifold with bounded geometry. Then, the high energy part:
\begin{equation*}
    \mathcal{O}_H: u \mapsto \int_1^\infty \Delta (\Delta+k^2)^{-1}u(\cdot) dk
\end{equation*}
satisfies
\begin{equation*}
    \|\mathcal{O}_H u\|_q \lesssim \|\nabla u\|_q, \quad \forall u\in C_c^\infty(M)
\end{equation*}
for all $1<q<\infty$.
\end{lemma}

\begin{proof}[Proof of Lemma~\ref{lemma_RRp_high_energy}]
The result is a direct consequence of duality and bounded geometry property. Let $v\in C_c^\infty(M)$ with $\|v\|_{q'}\le 1$. By the positivity and self-adjointness of $\Delta$, one has
\begin{align*}
    \langle \mathcal{O}_H u, v \rangle &= \left\langle \nabla u, \int_1^\infty \nabla (\Delta+k^2)^{-1}v dk \right\rangle\\
    &\le \|\nabla u\|_q \left\| \int_1^\infty \nabla (\Delta+k^2)^{-1}v dk   \right\|_{q'}.
\end{align*}
Note that the operator above applying to $v$ is nothing but the high energy part of the Riesz transform (see \cite[Proposition~5.1]{HS}). By \cite[Proposition 5.1]{HS}, one infers that
\begin{equation*}
    \left\| \int_1^\infty \nabla (\Delta+k^2)^{-1}v dk   \right\|_{q'}\lesssim \|v\|_{q'}
\end{equation*}
for all $1<q<\infty$. We mention that the proof of \cite[Proposition~5.1]{HS} is employable provided that $M$ has bounded geometry and $\Delta$ satisfies finite speed of propagation, see for example \cite{Si,S} for a formal definition. Next, by \cite{Si}, this finite speed propagation property is equivalent to the so-called Davies-Gaffney condition with respect to $\Delta$, which holds on any complete Riemannian manifolds. The result follows by ranging all $\|v\|_{q'}\le 1$.
\end{proof}

Next, we treat the low energy part $\mathcal{O}_L$. Observe that 
\begin{equation*}
    \Delta(\Delta+k^2)^{-1} = \textrm{Id} - k^2 (\Delta+k^2)^{-1}.
\end{equation*}
Hence, for $1<q<\infty$
\begin{align}\label{eq_O_L}
    \|\mathcal{O}_L\|_{q\to q} &= \left\| \int_0^1 \Delta (\Delta+k^2)^{-1} dk \right\|_{q\to q} \le \int_0^1 \left\| \Delta (\Delta+k^2)^{-1} \right\|_{q\to q} dk\\ \nonumber
    &\le 1 + \int_0^1 k^2 \|(\Delta+k^2)^{-1}\|_{q\to q} dk \lesssim 1,
\end{align}
where the last inequality follows by writing $(\Delta+k^2)^{-1} = \int_0^\infty e^{-t\Delta} e^{-tk^2}dt$ and using the contractivity of heat semigroup. Now, it is plain that by Lemma~\ref{lemma_RRp_high_energy} and \eqref{eq_O_L} that
\begin{align*}
    \|\Delta^{1/2}f_0\|_p \lesssim \|f_0\|_p + \|\nabla f_0\|_p \lesssim \|f\|_{L^p(B(o,2))} + \|\nabla f\|_p.
\end{align*}
Note that for $x\in B(o,2)$, one checks $1+r(x)\le 3$. It is then clear that by Lemma~\ref{lemma_513} $(\romannumeral2)$,
\begin{align}\label{eq_fff}
    \|f\|_{L^p(B(o,2))} \lesssim \left\| \frac{f(\cdot)}{1+r(\cdot)} \right\|_p \lesssim \|\nabla f\|_p,\quad \forall 1<p<\nu.
\end{align}
The claim \eqref{eq_claim_local} has been verified.

To this end, we need to confirm
\begin{equation*}
    \|\Delta^{1/2}f_1\|_p \lesssim \|\nabla f\|_p,\quad \forall p\quad \textrm{in \eqref{range}}.
\end{equation*}
It follows by a similar argument as in Section~\ref{section_4}, we have by Lemma~\ref{lemma_akcakn}
\begin{align*}
    |\mathcal{I}(f_1,g)| &\le \left|\int_M f_1(x) \int_M \int_0^\infty \eta_{\kappa^{-1}r(y)}(x) \Delta_x p_t(x,y) g(y) \frac{dt}{\sqrt{\pi t}} dy dx\right|\\
    &+ \left|\int_M f_1(x) \int_M \int_0^\infty \nabla_x [\eta_{\kappa^{-1}r(y)}(x)] \cdot \nabla p_t(x,y) g(y) \frac{dt}{\sqrt{\pi t}} dy dx\right|\\
    &\lesssim \int_M \frac{|f_1(x)|}{r(x)} |\mathcal{T}g(x)| dx.
\end{align*}
Note that for $x\in B(o,1)^c$, one has $1+r(x)\lesssim r(x)$. Thus, by Lemma~\ref{lemma_513} $(\romannumeral2)$ and Lemma~\ref{lemma_njcndjcn}, the above is bounded by
\begin{align*}
    \left\| \frac{f(\cdot)}{1+r(\cdot)} \right\|_p \|\mathcal{T}g\|_{p'} \lesssim \|\nabla f\|_p \|g\|_{p'}.
\end{align*}
It then follows by duality that
\begin{align*}
    \|\Delta^{1/2}f_1\|_p &\le \sup_{\|g\|_{p'}=1} |\langle \nabla f_1, R_1g \rangle| + \sup_{\|g\|_{p'}=1}|\langle \nabla f_1, R_2g \rangle| \\
    &\lesssim \sup_{\|g\|_{p'}=1} |\mathcal{I}(f_1,g)| + \sup_{\|g\|_{p'}=1} \|\nabla f_1\|_p \|R_2g\|_{p'}\\
    &\lesssim \|\nabla f\|_{p} + \|\nabla f_1\|_p \lesssim \|\nabla f\|_p + \|f\|_{L^p(B(o,2))}\\
    &\lesssim \|\nabla f\|_p,
\end{align*}
where the last inequality follows by \eqref{eq_fff}. Combining this with \eqref{eq_claim_local}, the proof of Theorem~\ref{thm_RR_QD} is complete.

\end{proof}

\begin{remark}
The proof of \eqref{eq_fff} is, in fact, a direct consequence of $p\textit{-}$hyperbolicity of the manifold (see, for example, \cite{Marc,Devyver_perturbation}) and this $p\textit{-}$hyperbolic property is an immediate consequence of $L^p$ Hardy's inequality.
\end{remark}

\section{Reverse inequality on Grushin spaces}\label{section_8}

In this section, we employ our method to the Grushin spaces. We prove that even without the assumption $(\textrm{P}_1)$, \eqref{eq_RRp} still holds for all $1<p<\infty$ on the Grushin spaces. Let $n\ge 2$, $m\ge 1$ and $\beta \ge 0$. Recall that the Grushin operator $L$ is defined by
\begin{align*}
    L = - \sum_{i=1}^n \partial_{x_i}^2 - |x|^{2\beta} \sum_{i=1}^m \partial_{y_i}^2 = \Delta_x + |x|^{2\beta} \Delta_y,\quad (x,y)\in \mathbb{R}^{n+m}.
\end{align*}
Define its associated gradient:
\begin{align}\label{Grushin_gradient}
   \nabla_{L} = \left(\nabla_x, |x|^{\beta} \nabla_y\right),
\end{align}
and the length of the gradient: $|\nabla_L f|^2 = \sum_{j=1}^n |\partial_{x_i}f|^2 + |x|^{2\beta} \sum_{j=1}^m |\partial_{y_j} f|^2$. We have integration by parts:
\begin{align*}
    \int_{\mathbb{R}^{n+m}} Lf(\xi) g(\xi) d\xi = \int_{\mathbb{R}^{n+m}} \nabla_L f(\xi) \cdot \nabla_L g(\xi) d\xi,
\end{align*}
for all $f,g\in C_c^\infty(\mathbb{R}^{n+m})$.

In this section, we prove the first part of Theorem~\ref{thm_RR_Grushin}. 

\begin{theorem}\label{thm_RR}
Let $n\ge 2$, $m\ge 1$ and $\beta > 0$. Then, the reverse Riesz inequality associated to $L$ given by \eqref{eq_grushin_operator} holds for all $1<p<\infty$ in the sense:
\begin{align*}
    \|L^{1/2} f\|_p \le C \|\nabla_L f\|_p, \quad 1<p<\infty
\end{align*}
for all $f\in C_c^\infty(\mathbb{R}^{n+m})$.
\end{theorem}

Denote by $e^{-tL}(\xi,\eta)$, the heat kernel of $L$. The Riesz transform associated to $L$ is then given by 
\begin{align*}
    \mathcal{R} = \nabla_L L^{-1/2} = \int_0^\infty \nabla_L e^{-tL} \frac{dt}{\sqrt{\pi t}}.
\end{align*}
In the rest of the paper, we use notations $\xi = (x,y)$ and $\eta = (x',y')$, where $x,x'\in \mathbb{R}^n$ and $y,y'\in \mathbb{R}^m$. Let $\kappa > 4$ be large. We split the kernel space into the following three regimes:
\begin{align*}
    \mathcal{D}_1 &= \{(\xi,\eta)\in \mathbb{R}^{n+m}\times \mathbb{R}^{n+m}; |x-x'| \ge \kappa^{-1} |x| \quad  \textrm{and}\quad |x'| \le \kappa |x| \},\\
    \mathcal{D}_2 &= \{(\xi,\eta)\in \mathbb{R}^{n+m}\times \mathbb{R}^{n+m}; |x-x'| \le \kappa^{-1} |x| \quad  \textrm{and}\quad |x'| \le \kappa |x| \},\\
    \mathcal{D}_3 &= \{(\xi,\eta)\in \mathbb{R}^{n+m}\times \mathbb{R}^{n+m}; |x'| \ge \kappa |x| \}.
\end{align*}
For $i=1,2,3$, we set
\begin{align*}
    \mathcal{R}_i(\xi,\eta) := \nabla_L L^{-1/2}(\xi,\eta) \chi_{\mathcal{D}_i} =  \int_0^\infty \nabla_L e^{-tL}(\xi,\eta) \chi_{\mathcal{D}_i} \frac{dt}{\sqrt{\pi t}}.
\end{align*}
We call $\mathcal{R}_1$ the good part of the Riesz transform, $\mathcal{R}_2$ the diagonal part, and $\mathcal{R}_3$ the bad part.

\subsection{Estimates for the geometry}

In this subsection, we obatin and recall some geometric properties on the Grushin spaces. Consider function space
\begin{align*}
    D = \left\{\phi \in W^{1,\infty}(\mathbb{R}^{n+m}); \sum_{1\le i,j\le n+m} A_{ij} \partial_i \phi \partial_j \phi \le 1  \right\},
\end{align*}
where
\begin{align}
A_{ij} = 
\begin{bmatrix}
    I_n & 0 \\
    0 & |x|^{\beta} I_m
\end{bmatrix},
\end{align}
and $I_n$ refers to the $n\times n$ identity matrix. Then, the canonical distance on $\mathbb{R}^{n+m}$ is defined by
\begin{align*}
    d(\xi,\eta) := \sup_{\phi \in D} \left|\phi(\xi) - \phi(\eta)\right| \in [0,\infty]
\end{align*}
for all $\xi,\eta \in \mathbb{R}^{n+m}$, and its associated geodesic balls can be then described by the set
\begin{align*}
    B(\xi,r) = \left\{ \eta\in \mathbb{R}^{n+m}; d(\eta,\xi) < r \right\}.
\end{align*}
We set $V(\xi,r)$ to be the volume of the ball $B(\xi,r)$ with respect to the usual Lebesgue measure in $\mathbb{R}^{n+m}$, i.e. $|B(\xi,r)|$. First of all, let us recall the following volume and distance estimates from \cite{DS1}.

\begin{lemma}\cite[Proposition~5.1]{DS1}\label{lemma_DS1}
Under the assumptions of Theorem~\ref{thm_RR_Grushin}, the follwoing estimates hold:
\begin{align}
    d(\xi,\eta) = d((x,y);(x',y')) \sim |x-x'| + \frac{|y-y'|}{(|x|+|x'|)^{\beta} + |y-y'|^{\frac{\beta}{\beta+1}}},
\end{align}
and
\begin{align*}
    V(\xi,r) = V((x,y),r) \sim \begin{cases}
        r^{\mathcal{Q}}, & r\ge |x|,\\
        r^{n+m}|x|^{m\beta}, & r\le |x|,
    \end{cases}
\end{align*}
where $\mathcal{Q} = n + m(\beta+1)$ is the so-called homogeneous dimension.

Moreover, the following volume doubling condition holds, i.e.
\begin{align*}
    V(\xi,sr) \le C s^{\mathcal{Q}} V(\xi,r)
\end{align*}
for all $\xi \in \mathbb{R}^{n+m}$ and all $s\ge 1$, $r>0$.
\end{lemma}

Compared to \cite{FL}, we develop the following properties. We use notations $B_n(x,r)$ and $B_m(y,r)$ to denote the Euclidean balls in $\mathbb{R}^n$ and $\mathbb{R}^m$ respectively.

\begin{lemma}\label{lemma_G_volume}
For $r>0$ and $(0,y)\in \mathbb{R}^{n+m}$, there exists constants $c_1,c_2>0$ such that
\begin{align}\label{eq_lemma_volume_1}
    B_n(0,c_1 r) \times B_m \left(y,c_1 r^{\beta+1}\right) \subset B\left((0,y), r\right) \subset B_n(0,c_2 r) \times B_m \left(y,c_2r^{\beta+1}\right) .
\end{align}
Moreover, there exists $c_3,c_4>0$ such that for $(x,y)\in \mathbb{R}^{n+m}$,
\begin{align}\label{eq_lemma_volume_2}
    B_n(x,c_3|x|) \times B_m\left(y,c_3|x|^{\beta+1}\right) \subset B((x,y), |x|) \subset B_n(x,c_4|x|) \times B_m\left(y,c_4|x|^{\beta+1}\right).
\end{align}
\end{lemma}

\begin{proof}[Proof of Lemma~\ref{lemma_G_volume}]
We begin with the forward direction of \eqref{eq_lemma_volume_1}, i.e. 
\begin{align*}
    B_n(0,c_1 r) \times B_m\left(y,c_1 r^{\beta+1}\right) \subset B((0,y), r).
\end{align*}
Let $(x',y')\in B_n(0,r) \times B_m\left(y,r^{\beta+1}\right)$. Then, $|x'|\le r$ and $|y-y'|\le r^{\beta+1}$. Next, by Lemma~\ref{lemma_DS1}, 
\begin{align*}
    d\left((0,y), (x',y')\right) \sim |x'| + \frac{|y-y'|}{|x'|^{\beta} + |y-y'|^{\frac{\beta}{\beta+1}}}.
\end{align*}
Therefore, it suffices to show 
\begin{align}\label{eq_lemma0}
    \frac{|y-y'|}{|x'|^{\beta} + |y-y'|^{\frac{\beta}{\beta+1}}} \lesssim r.
\end{align}
To see this, if $|y-y'|\ge |x'|^{\beta+1}$, then the LHS of \eqref{eq_lemma0} is comparable to $|y-y'|^{\frac{1}{\beta+1}}$, which is bounded by $r$. On the other hand, if $|y-y'|\le |x'|^{\beta+1}$, then the LHS of \eqref{eq_lemma0} is comparable to $|x'|^{-\beta}|y-y'|$, which is bounded by $|x'|\le r$. This completes the forward direction of \eqref{eq_lemma_volume_1}.

For the reverse direction of \eqref{eq_lemma_volume_1}, i.e. 
\begin{align*}
    B\left((0,y), r\right) \subset B_n(0,c_2 r) \times B_m\left(y,c_2r^{\beta+1}\right).
\end{align*}
Let $(x',y')\in B\left((0,y), r\right)$. Then, by Lemma~\ref{lemma_DS1}, \eqref{eq_lemma0} holds. Note that $x'\in B_n(0,cr)$ is clear. To this end, one argues as before, if $|y-y'|\ge |x'|^{\beta+1}$, then
\begin{align*}
    \eqref{eq_lemma0} \lesssim r \iff |y-y'|^{\frac{1}{\beta+1}} \lesssim r \iff y'\in B_m\left((y,cr^{\beta+1}\right).
\end{align*}
If $|y-y'|\le |x'|^{\beta+1}$, one deduces
\begin{align*}
    \eqref{eq_lemma0} \iff |x'|^{-\beta}|y-y'| \lesssim r \iff |y-y'| \lesssim r |x'|^\beta \lesssim r^{\beta+1}
\end{align*}
as desired. The inequality \eqref{eq_lemma_volume_2} can be proved similarly and we omit the details.
\end{proof}

\begin{lemma}\cite[Theorem~6.4, Corollary~6.6]{DS1}\label{le_RS}
Under the assumptions of Theorem~\ref{thm_RR_Grushin}, the heat kernel of the Grushin operator $L$ satisfies 
\begin{align*}
    e^{-tL}(\xi,\eta) \le \frac{C}{V(\xi, \sqrt{t})} e^{-\frac{d(\xi,\eta)^2}{ct}}
\end{align*}
for all $t>0$ and almost all $\xi,\eta \in \mathbb{R}^{n+m}$.
\end{lemma}

To proceed our argument, let us recall the following gradient estimates from \cite[Lemma~5.6]{H3}.

\begin{lemma}\cite[Lemma~5.6]{H3}\label{le_Grushin_gradient}
Under the assumptions of Theorem~\ref{thm_RR_Grushin}, the following gradient estimate holds:
\begin{align*}
    \left|\nabla_{L}e^{-tL}(\xi,\eta)\right|\le \left(\frac{1}{\sqrt{t}} + \frac{1}{|x|}\right) \frac{C}{V(\xi,\sqrt{t})}e^{-\frac{d(\xi,\eta)^2}{ct}},
\end{align*}
where $\xi = (x,y) \in \mathbb{R}^n \setminus \{0\} \times \mathbb{R}^m$.
\end{lemma}

\subsection{Estimates for the good part}

In this subsection, we estimate the $L^p$-boundedness of the good part of the Riesz transform. That is we prove the following result.

\begin{proposition}\label{prop_R1}
Under the assumptions of Theorem~\ref{thm_RR_Grushin}, $\mathcal{R}_1$ is bounded on $L^p$ for all $\frac{n}{n-1} < p < \infty$ in the sense:
\begin{align*}
    \|\mathcal{R}_1 f\|_p \le C \|f\|_p, \quad 2<p<\infty
\end{align*}
for all $f\in C_c^\infty(\mathbb{R}^n \setminus \{0\} \times \mathbb{R}^m)$.
\end{proposition}

The first step is to obtain a suitable estimate for $\mathcal{R}_1(\xi,\eta)$.

\begin{lemma}\label{lemma_kernel_R1}
Under the assumptions of Theorem~\ref{thm_RR}, we have

\begin{align*}
|\mathcal{R}_1(\xi,\eta)| \lesssim \begin{cases}
    |B_n(0,|x|)|^{-1} |B_m\left(y,|x|^{\beta+1}\right)|^{-1} \chi_{B_n(0,\kappa |x|)}(x'), & |y-y'|\le |x|^{\beta+1},\\
    |y-y'|^{-\frac{\mathcal{Q}}{\beta+1}} \chi_{B_n(0,\kappa |x|)}(x') + |x|^{-1} |y-y'|^{-\frac{\mathcal{Q}-1}{\beta+1}} \chi_{B_n(0,\kappa |x|)}(x'), & |y-y'|\ge |x|^{\beta+1}.
\end{cases}
\end{align*}

\end{lemma}

\begin{proof}[Proof of Lemma~\ref{lemma_kernel_R1}]
Write
\begin{align*}
    |\mathcal{R}_1(\xi,\eta)|\le \left(\int_0^{|x|^2} + \int_{|x|^2}^\infty\right) \left|\nabla_L e^{-tL}(\xi,\eta)\right| \chi_{\mathcal{D}_1}(\xi,\eta) \frac{dt}{\sqrt{\pi t}}:= I_1 + I_2.
\end{align*}
Observe that for $(\xi,\eta)\in \mathcal{D}_1$, we have by Lemma~\ref{le_Grushin_gradient} and Lemma~\ref{lemma_DS1},
\begin{align*}
    I_1 &\lesssim \int_0^{|x|^2} \frac{1}{|x|^{m\beta} (\sqrt{t})^{n+m}} e^{-\frac{d(\xi,\eta)^2}{ct}} \frac{dt}{t}\\
    &\lesssim |x|^{-m\beta} d(\xi,\eta)^{-n-m} e^{-c \frac{d(\xi,\eta)^2}{|x|^2}}\\
    &\lesssim |x|^{-m\beta+\epsilon} d(\xi,\eta)^{-n-m-\epsilon},\quad \forall \epsilon\ge 0\\
    &\lesssim d(\xi,\eta)^{-\mathcal{Q}},
\end{align*}
where we choose $\epsilon = m\beta$.

On the other hand, by a similar argument, for $(\xi,\eta)\in \mathcal{D}_1$,
\begin{align*}
    I_2 &\lesssim |x|^{-1} \int_{|x|^2}^\infty t^{-\frac{\mathcal{Q}+1}{2}} e^{-\frac{d(\xi,\eta)^2}{ct}} dt \lesssim |x|^{-1} d(\xi,\eta)^{-\mathcal{Q}+1}.
\end{align*}
Next, note that for $(\xi,\eta)\in \mathcal{D}_1$, we have $|x-x'|\sim |x|$ and $|x|+|x'|\sim |x|$. Hence by Lemma~\ref{lemma_DS1} again,
\begin{align*}
    d(\xi,\eta) \sim |x| + \frac{|y-y'|}{|x|^\beta + |y-y'|^{\frac{\beta}{\beta+1}}} \sim \begin{cases}
        |x|, & |y-y'|\le |x|^{\beta+1},\\
        |y-y'|^{\frac{1}{\beta+1}}, & |y-y'|\ge |x|^{\beta+1}.
    \end{cases}
\end{align*}
The result follows immediately.

\end{proof}

By the above Lemma~\ref{lemma_kernel_R1}, we get pointwise estimate:
\begin{align*}
    |\mathcal{R}_1f(\xi)| \lesssim T_1f(\xi) + T_2f(\xi),
\end{align*}
where
\begin{align*}
    T_1: f \mapsto \frac{1}{V_n(0,|x|) V_m(y,|x|^{\beta+1})} \int_{B_n(0,\kappa|x|)} \int_{B_m(y,|x|^{\beta+1})} |f(x',y')| dy' dx',
\end{align*}
and
\begin{align*}
    T_2: f \mapsto |x|^{-\epsilon} \int_{B_n(0,\kappa|x|)} \int_{B_m(y,|x|^{\beta+1})^c} \frac{|f(x',y')|}{|y-y'|^{\frac{\mathcal{Q}-\epsilon}{\beta+1}}} dy' dx',\quad \epsilon = 0,1.
\end{align*}

For the rest of the article, we use $\mathcal{M}$ to denote the Hardy-Littlewood maximal operator on the Grushin space, i.e.
\begin{align*}
    \mathcal{M}u(\xi) = \sup_{B\ni \xi} \frac{1}{|B|} \int_B |u(\eta)| d\eta,\quad u\in L_{\textrm{loc}}^1(\mathbb{R}^{n+m}),
\end{align*}
where $B \subset \mathbb{R}^{n+m}$ is the geodesic ball with respect to the Riemannian distance introduced before. Since the Grushin space satisfies doubling condition \eqref{Doubling}, $\mathcal{M}$ acts as a bounded operator on $L^p$ for all $1<p\le \infty$.

\begin{lemma}\label{lemma_R1}
$T_1$ is bounded on $L^p$ for all $1<p\le \infty$.
\end{lemma}

\begin{proof}[Proof of Lemma~\ref{lemma_R1}]
Apparently, by Lemma~\ref{lemma_G_volume},
\begin{align*}
    |T_1f(\xi)|\lesssim \frac{1}{V((0,y),|x|)} \int_{B\left((0,y),c_2 \kappa|x|\right)} |f(\eta)| d\eta
    \lesssim \mathcal{M}f(\xi).
\end{align*}
The result follows by maximal theorem.

\end{proof}

\begin{lemma}\label{lemma_good2}
$T_2$ is bounded on $L^p$ for all $\frac{n}{n-1}<p<\infty$.
\end{lemma}

\begin{proof}[Proof of Lemma~\ref{lemma_good2}]
Note that by dyadic decomposition,
\begin{align*}
    |T_2&f(\xi)| \lesssim |x|^{-\epsilon} \int_{|x'|\le \kappa|x|} \sum_{j=0}^\infty \left(2^j |x|^{\beta+1} \right)^{-\frac{\mathcal{Q}-\epsilon}{\beta+1}} \int_{B_m\left(y,2^{j+1}|x|^{\beta+1}\right)\setminus B_m\left(y,2^{j}|x|^{\beta+1}\right)} |f(\eta)|d\eta\\
    &\lesssim |x|^{-\mathcal{Q}} \sum_{j=0}^\infty 2^{-j\frac{\mathcal{Q}-\epsilon}{\beta+1}} \int_{B_n(0,\kappa|x|)}  \left(\int_{B_m\left(y,2^{j+1}|x|^{\beta+1}\right)} |f(x',y')|^p dy'\right)^{\frac{1}{p}} \left(2^j |x|^{\beta+1}\right)^{\frac{m}{p'}} dx'\\
    &\lesssim |x|^{-\mathcal{Q}} \sum_{j=0}^\infty 2^{-j\frac{\mathcal{Q}-\epsilon}{\beta+1}} 2^{j\frac{m}{p'}} |x|^{\frac{m(\beta+1)}{p'}} |x|^{\frac{n}{p'}} \left(\int_{B_n(0,\kappa|x|)\times B_m\left(y,\kappa^{\beta+1}2^{j+1}|x|^{\beta+1}\right)} |f(\eta)|^p  d\eta \right)^{\frac{1}{p}}\\
    &\lesssim |x|^{-\frac{\mathcal{Q}}{p}} \sum_{j=0}^\infty 2^{-j\left(\frac{\mathcal{Q}-\epsilon}{\beta+1} - \frac{m}{p'}\right)} \left(\frac{1}{V\left((0,y), \kappa 2^{\frac{j+1}{\beta+1}}|x|\right)}\int_{B\left((0,y), \kappa 2^{\frac{j+1}{\beta+1}}|x|\right)} |f(\eta)|^p d\eta\right)^{\frac{1}{p}} \left(2^{\frac{j+1}{\beta+1}}|x|\right)^{\frac{\mathcal{Q}}{p}}\\
    &\lesssim \mathcal{M}\left(|f|^p\right)(\xi)^{\frac{1}{p}} \sum_{j=0}^\infty 2^{-j\left(\frac{\mathcal{Q}-\epsilon}{\beta+1} - \frac{m}{p'}-\frac{\mathcal{Q}}{p(\beta+1)}\right)}  \\
    &\lesssim \mathcal{M}\left(|f|^p\right)(\xi)^{\frac{1}{p}},
\end{align*}
provided $p>n'$ if $\epsilon=1$ and $p>1$ if $\epsilon=0$.

To this end, for any $n'<p<q<\infty$, we easily deduce by maximal theorem that
\begin{align*}
    \|T_2f\|_q^q \lesssim \left\|\mathcal{M}\left(|f|^p\right)\right\|_{q/p}^{q/p} \lesssim \|f\|_q^q 
\end{align*}
as desired.

\end{proof}

Proposition~\ref{prop_R1} is then followed by Lemma~\ref{lemma_R1}, Lemma~\ref{lemma_good2}.

\subsection{Estimates for the diagonal part}
Next, we consider $\mathcal{R}_2$, where
\begin{align*}
    \mathcal{R}_2(\xi,\eta) = \int_0^\infty \nabla_L e^{-tL}(\xi,\eta) \chi_{\{|x-x'|\le \kappa^{-1} |x|\}} \frac{dt}{\sqrt{\pi t}}.
\end{align*}

In this subsection, we prove the following result.

\begin{proposition}\label{prop_R2}
Under the assumptions of Theorem~\ref{thm_RR}, $\mathcal{R}_2$ is bounded on $L^p$ for all $2<p<\infty$ in the sense:
\begin{align*}
    \|\mathcal{R}_2f\|_p \le C \|f\|_p, \quad 2<p<\infty
\end{align*}
for all $f\in C_c^\infty(\mathbb{R}^n \setminus \{0\} \times \mathbb{R}^m)$.
\end{proposition}

Before presenting the proof, we recall the following lemma from \cite[Section~4.3]{Gilles}; see also \cite[Section~2.1]{DR_hardy}. Since some variational properties are needed, we give a proof for the sake of completeness.

\begin{lemma}\label{lemma_cover}
Let $M$ be a complete Riemannian manifold satisfying \eqref{Doubling}. Let $o\in M$ be a reference point. Suppose $R>0$. Then, there exists a sequence of balls $\{B_\alpha = B(x_\alpha, r_\alpha)\}_{\alpha\ge 0}$ such that 

$(1)$ $M = \cup_{\alpha\ge 0} B_\alpha$, where $B_0 = B(o,R)$ and $B_\alpha$ ($\alpha\ge 1$) is remote,


$(2)$ $   \sum_\alpha \chi_{B_\alpha}(x) \le c_0$ for all $x\in M$, where $c_0>0$ does not depend on $R$,

$(3)$ for all $\alpha \ne 0$, $2^{-10}r(x_\alpha)\le r_\alpha \le 2^{-9} r(x_\alpha)$.

\end{lemma}

\begin{proof}[Proof of Lemma~\ref{lemma_cover}]
Set $B_0 = B(o,R)$ and $A_N:= B(o, R 2^{N}) \setminus B(o, R 2^{N-1})$ for each $N\ge 1$. Apparently, 
\begin{align*}
    M = B_0 \cup \bigcup_{N\ge 1}A_N,\quad A_N \subset \bigcup_{x\in A_N} B(x, R 2^{N-13}).
\end{align*}
By \cite[Theorem~1.2]{Juha}, we find an 
index set $I_N$ and a collection of balls $\{B(x_{N,j},R2^{N-13})\}_{j\in I_N}$ with $x_{N,j}\in A_N$, pairwise disjoint and $A_N \subset \bigcup_{j\in I_N}B(x_{N,j}, R2^{N-10})$. By \eqref{Doubling}, one deduces the finiteness of $\# I_N$:
\begin{align*}
    \# I_N V(o,R 2^{N+2}) \le \sum_{j \in I_N} V(x_{N,j}, R2^{N+3}) \lesssim \sum_{j \in I_N} V(x_{N,j}, R2^{N-13}) \le V(o,R2^{N+2}),
\end{align*}
where the implicit constant does not depend on $N$. Note that by setting $B_{\alpha} = B(x_{\alpha}, r_\alpha)$ with $x_\alpha = x_{N,j}$, $r_\alpha = R 2^{N-10}$ and relabeling, we construct a sequence of balls $\{B_\alpha\}_{\alpha \ge 0}$ with $\cup_{\alpha \ge 0}B_\alpha = M$. Moreover, we have for $\alpha \ne 0$ (since $x_\alpha \in A_N$),
\begin{align*}
2^{-10}r(x_\alpha)\le r_\alpha \le 2^{-9} r(x_\alpha),
\end{align*}
and for all $x\in B_\alpha$,
\begin{align*}
    2^8 r_\alpha \le r(x) \le 2^{11}r_\alpha.
\end{align*}
This proves $(1)$ and $(3)$.

Note that this construction guarantees that $B_0 = B(o,R)$ is the only ball in the collection which contains $o$. In addition, for $x\in M \setminus \{o\}$, one sets $J_x = \{\alpha; x\in B_\alpha\}$. Observe that if $x\in A_N$ for some $N\ge 1$ (the case $x\in B_0$ is similar), and $x\in B_\alpha$ for some $\alpha$, then $\alpha$ is either in $I_N$, $I_{N-1}$ or $I_{N+1}$. Therefore, by \eqref{Doubling}, we have
\begin{align*}
    \# J_x V(x,2^{-1}r(x)) &\le \sum_{\alpha \in J_x} V(x_\alpha, 2^{-1}r(x)+2^{-8}r(x)) \lesssim \sum_{\alpha\in J_x} V(x_\alpha, 2^{-20}r(x))\\
    &\le \sum_{\alpha\in J_x} V(x_\alpha, 2^{-3}r_\alpha) = \sum_{l\in \{-1,0,1\}} \sum_{\alpha \in J_x \cap I_{N+l}} V(x_\alpha, 2^{-3}r_\alpha) \\
    &= \sum_{l\in \{-1,0,1\}} \textrm{vol}\left(\cup_{\alpha \in J_x\cap I_{N+l}}B(x_\alpha, 2^{-3} r_\alpha)\right)\\
    &\le 3V(x, 2^{-1}r(x)).
\end{align*}
since for all $\alpha \in J_x\cap I_{N+l}$, $B(x_\alpha, r_\alpha/8) \subset B(x,9r_\alpha/8) \subset B(x, \frac{9 }{8}2^{-8} r(x)) \subset B(x, 2^{-1}r(x))$. The case $x\in B_0$ is similar. Hence, we conclude that there exists a constant $c_0>0$ which does not depend on $R,N$ such that
\begin{align*}
    \sum_\alpha \chi_{B_\alpha}(x) \le c_0,\quad \forall x\in M.
\end{align*}
This completes the proof of $(2)$ and hence the Lemma~\ref{lemma_cover}.

\end{proof}

Recall that by our assumption, $f\in C_c^\infty(\mathbb{R}^n \setminus \{0\} \times \mathbb{R}^m)$. For a moment, let us further assume that there exists a $\delta>0$ such that $\textrm{supp}(f) \subset B_n(0,\delta)^c \times \mathbb{R}^m$. By choosing $M=\mathbb{R}^n$ and $R=\delta$ in Lemma~\ref{lemma_cover}, we may decompose $\mathbb{R}^n = B_n^0 \cup \left(\cup_{\alpha \ge 1}B_n^\alpha\right)$, where $B_n^0 = B_n(0,\delta)$ and $B_n^\alpha = B_n(x_\alpha, r_\alpha)$ is remote for $\alpha \ge 1$. Let $\{\mathcal{X}_\alpha\}$ be a smooth partition of unity subordinate to this covering. In particular, $\textrm{supp}\left(\mathcal{X}_\alpha\right) \subset B_n^\alpha$. Therefore, by choosing $\kappa $ (not depending on $\delta$) large enough, the support property of $\mathcal{R}_2$ implies
\begin{align*}
    |\mathcal{R}_2f(\xi)| \le \sum_{\alpha \ge 0} \left|\chi_{4B_n^\alpha}(x) \mathcal{R}(f\mathcal{X}_{\alpha})(\xi)\right|.
\end{align*}
Set $\mathcal{R}_\alpha := \chi_{4B_n^\alpha} \mathcal{R} \mathcal{X}_{\alpha}$. Note that by the support property of $f$, the first term in the above sum vanishes, i.e. $\mathcal{R}_0 f = 0$. By finite overlap, it suffices to show
\begin{align*}
    \sup_{\alpha \ge 1} \|\mathcal{R}_\alpha\|_{p\to p} \lesssim 1,\quad \forall 2<p<\infty,
\end{align*}
where the implicit constant does not depend on $\delta$. Finally, we conclude the proof by letting $\delta \to 0$.

For each $\alpha \ge1$, we further split
\begin{align*}
    \mathcal{R}_\alpha f(\xi) = \chi_{4B_n^\alpha}(x) &\int_0^{r_\alpha^2} \nabla_L e^{-tL}(f\mathcal{X}_{\alpha})(\xi) \frac{dt}{\sqrt{t}} \\
    &+ \chi_{4B_n^\alpha}(x) \int_{r_\alpha^2}^\infty \nabla_L e^{-tL}(f\mathcal{X}_{\alpha})(\xi) \frac{dt}{\sqrt{t}}:= \mathcal{R}_{\alpha,1}f + \mathcal{R}_{\alpha,2}f.
\end{align*}

\begin{lemma}\label{lemma_R2_highenergy}
$\sup_{\alpha\ge 1} \|\mathcal{R}_{\alpha,2}\|_{p\to p} \lesssim 1$ for all $\frac{n}{n-1}<p<\infty$.
\end{lemma}

\begin{proof}[Proof of Lemma~\ref{lemma_R2_highenergy}]
Note that for $x\in 4B_n^\alpha$ and $x'\in B_n^\alpha$, we have $|x|,|x'| \sim r_\alpha$. Hence, by the gradient estimates,
\begin{align*}
    |\nabla_L e^{-tL}(\xi,\eta)|\lesssim \left(\frac{1}{\sqrt{t}} + \frac{1}{r_\alpha}\right) \frac{1}{V(\xi,\sqrt{t})} e^{-c\frac{\sigma^2}{t}}, 
\end{align*}
where 
\begin{align}\label{eq_sigma}
    \sigma = \frac{|y-y'|}{r_\alpha^{\beta} + |y-y'|^{\frac{\beta}{\beta+1}}}.
\end{align}
Therefore, a straightforward calculation yields
\begin{align*}
    |\mathcal{R}_{\alpha,2}(\xi,\eta)|\lesssim r_\alpha^{-1-\epsilon} \sigma^{-\mathcal{Q}+1+\epsilon} \chi_{4B_n^\alpha}(x) \chi_{B_n^\alpha}(x')
\end{align*}
for any $0\le \epsilon < \mathcal{Q}-1$.


Observe that
\begin{align*}
    \sigma \sim \begin{cases}
        r_\alpha^{-\beta} |y-y'|, & |y-y'|\le r_\alpha^{\beta+1},\\
        |y-y'|^{\frac{1}{\beta+1}}, & |y-y'|\ge r_\alpha^{\beta+1}.
    \end{cases}
\end{align*}
We obtain upper bound:
\begin{align*}
    |\mathcal{R}_{\alpha,2}f(\xi)| \lesssim \chi_{4B_n^\alpha}(x) r_\alpha^{-1-\epsilon} \int_{B_n^\alpha} \int_{B_m \left(y,r_\alpha^{\beta+1}\right)} \left(r_\alpha^{-\beta}|y-y'|\right)^{-\mathcal{Q}+1+\epsilon} |f(x',y')| dy' dx'\\
    + \chi_{4B_n^\alpha}(x) r_\alpha^{-1} \int_{B_n^\alpha} \int_{B_m \left(y,r_\alpha^{\beta+1}\right)^c} |y-y'|^{-\frac{\mathcal{Q}-1}{\beta+1}} |f(x',y')| dy' dx'\\
    := If(\xi) + IIf(\xi).
\end{align*}
Note that by Lemma~\ref{lemma_good2}, replacing $|x|$ by $r_\alpha$, and the fact $B_n^\alpha \subset B_n(0,cr_\alpha)$ for some $c>0$, we have $|II f(\xi)|\lesssim \mathcal{M}(|f|^p)(\xi)^{\frac{1}{p}}$. Moreover, $\|II\|_{q\to q}\lesssim 1$ for all $q>n'$ and all $\alpha > 0$. Next, we treat $I$. By Hölder's inequality, for any $p>1$,
\begin{align*}
    |If(\xi)|&\le r_\alpha^{-1-\epsilon} \left(\int_{B_n(0,c r_\alpha) \times B_m \left(y,cr_\alpha^{\beta+1}\right)} |f(\eta)|^p d\eta \right)^{\frac{1}{p}} \left(  \int_{B_n^\alpha} \int_{B_m \left(y,r_\alpha^{\beta+1}\right)} \frac{ r_\alpha^{p' \beta(\mathcal{Q}-1-\epsilon)} dy' dx'}{|y-y'|^{p'(\mathcal{Q}-1-\epsilon)}} \right)^{\frac{1}{p'}}\\
    &\lesssim r_\alpha^{-1-\epsilon + \frac{\mathcal{Q}}{p}} \mathcal{M}(|f|^p)(\xi)^{\frac{1}{p}} r_\alpha^{\frac{n}{p'}+\beta(\mathcal{Q}-1-\epsilon)} \left( \int_0^{r_\alpha^{\beta+1}} s^{m-p'(\mathcal{Q}-1-\epsilon)-1} ds \right)^{\frac{1}{p'}}\lesssim \mathcal{M}(|f|^p)(\xi)^{\frac{1}{p}},
\end{align*}
where we choose $\mathcal{Q}-1-\frac{m}{p'} < \epsilon < \mathcal{Q}-1$. By maximal theorem, $\sup_\alpha \|I\|_{q\to q}\lesssim 1$ for all $1<q<\infty$. This completes the proof of Lemma~\ref{lemma_R2_highenergy}. 
\end{proof}

Next, we consider $\mathcal{R}_{\alpha,1}$. We further split
\begin{align*}
    \mathcal{R}_{\alpha,1}(\xi,\eta) = \mathcal{R}_{\alpha,1}(\xi,\eta)\chi_{\{|y-y'|\le r_\alpha^{\beta+1}\}} + \mathcal{R}_{\alpha,1}(\xi,\eta)\chi_{\{|y-y'|\ge r_\alpha^{\beta+1}\}}:= \mathcal{R}_{\alpha,1}^1 + \mathcal{R}_{\alpha,1}^2
\end{align*}
for all $\alpha \ge 1$.

\begin{lemma}\label{lemma_R2_lowenergy_error}
$\sup_{\alpha\ge 1} \|\mathcal{R}_{\alpha,1}^2\|_{p\to p} \lesssim 1$ for all $1<p<\infty$.
\end{lemma}

\begin{proof}[Proof of Lemma~\ref{lemma_R2_lowenergy_error}]
Note that for $x\in 4B_n^\alpha$, $x' \in B_n^\alpha$, $t\le r_\alpha^2$ 

\begin{align*}
    |\mathcal{R}_{\alpha,1}^2(\xi,\eta)| \lesssim \int_0^{r_\alpha^2} \frac{1}{\sqrt{t}V(\xi,\sqrt{t})} e^{-c\frac{\sigma^2}{t}} \frac{dt}{\sqrt{t}},
\end{align*}
where $\sigma$ is given by \eqref{eq_sigma}. By geometry, $V(\xi,\sqrt{t}) \sim |x|^{m\beta}t^{\frac{n+m}{2}} \sim r_\alpha^{m\beta} t^{\frac{n+m}{2}}$. In addition, since $|y-y'|\ge r_\alpha^{\beta+1}$, we have $\sigma \sim |y-y'|^{\frac{1}{\beta+1}}$. Therefore, the kernel is bounded by
\begin{align*}
    \chi_{4B_n^\alpha}(x) \chi_{B_n^\alpha}(x') r_\alpha^{-m\beta}& \chi_{|y-y'|\ge r_\alpha^{\beta+1}}(y') |y-y'|^{-\frac{n+m}{\beta+1}} e^{-c\frac{\sigma^2}{r_\alpha^2}} \\&\lesssim\chi_{4B_n^\alpha}(x) \chi_{B_n^\alpha}(x')\chi_{|y-y'|\ge r_\alpha^{\beta+1}}(y') |y-y'|^{-\frac{\mathcal{Q}}{\beta+1}}.
\end{align*}
Thus,
\begin{align*}
    |\mathcal{R}_{\alpha,1}^2f(\xi)|\lesssim \chi_{4B_n^\alpha}(x) \int_{B_n(0,cr_\alpha)} \int_{B_m\left(y,r_\alpha^{\beta+1}\right)^c} |y-y'|^{-\frac{\mathcal{Q}}{\beta+1}} |f(\eta)| d\eta.
\end{align*}
By Lemma~\ref{lemma_good2}, replacing $|x|$ by $r_\alpha$ and $\epsilon =0$, we get
\begin{align*}
    \sup_{\alpha \ge 1} \|\mathcal{R}_{\alpha,1}^2\|_{p\to p} \lesssim 1
\end{align*}
for all $1<p<\infty$. 

\end{proof}

Next, we treat $\mathcal{R}_{\alpha,1}^1$. Recall the Hardy-Littlewood maximal operator relative to some measurable subset $E$ of $\mathbb{R}^{n+m}$. For $\xi\in \mathbb{R}^{n+m}$, and $f\in L_{\textrm{loc}}^1$,
\begin{align*}
    \mathcal{M}_Ef(\xi) = \sup_{B \ni \xi} \frac{1}{|B \cap E|} \int_{B \cap E} |f(\eta)| d\eta.
\end{align*}
Since the Grushin space $(\mathbb{R}^{n+m},g)$ satisfies \eqref{Doubling}, the sublinear operator $\mathcal{M}_E$ is of weak type $(1,1)$ and bounded on $L^p$ for all $1<p\le \infty$ with operator norm not depending on $|E|$.

\begin{theorem}\cite[Theorem~2.4]{ACDH}\label{thm_ACDH}
Let $(M,d,\mu)$ be a metric space satisfying doubling condition. Suppose that $T$ is a bounded sublinear operator which is bounded on $L^2(M,\mu)$, and let $A_r$, $r>0$, be a family of linear operators acting on $L^2(M,\mu)$. Let $p_0\in (2,\infty]$. Let $E_1,E_2$ be two subsets of $M$ such that $E_1\subset E_2$, $\mu(E_2)<\infty$.  Assume
\begin{align}\label{thm_ACDH1}
    \mathcal{M}^{\#}_{E_2}f(x)^2 := \sup_{B\ni x}\frac{1}{\mu(B\cap E_2)} \int_{B\cap E_2} |T(I-A_{r})f|^2,
\end{align}
is bounded from $L^p(E_1) \to L^p(E_2)$ for all $p\in (2,p_0)$, and for some sublinear operator $S$ bounded from $L^p(E_1) \to L^p(E_2)$ for all $p\in (2,p_0)$,
\begin{align}\label{thm_ACDH2}
    \left(\frac{1}{\mu(B\cap E_2)}\int_{B\cap E_2} |TA_{r}f|^{p_0} d\mu\right)^{\frac{1}{p_0}} \le C_2 \mathcal{M}_{E_2}(|Tf|^2)^{\frac{1}{2}}(x) + C_2 Sf(x),
\end{align}
for all $f\in L^2$ supported in $E_1$, all balls $B\subset M$ and all $x\in B\cap E_2$, where $r=r(B)$ is the radius of $B$. If $2<p<p_0$ and $Tf\in L^p(E_2)$ whenever $f\in L^p(E_1)$, then $T$ is bounded from $L^p(E_1)$ to $L^p(E_2)$ and its operator norm is bounded by a constant depending only on the operator norm of $T$ on $L^2$, the $L^p(E_1)\to L^p(E_2)$ operator norms of $\mathcal{M}_{E_2}^{\#}$ and $S$, the doubling constant from \eqref{Doubling}, $p,p_0$, and $C_1,C_2$. 
\end{theorem}

\begin{remark}
In the above theorem, for the case $p_0=\infty$, the assumption \eqref{thm_ACDH2} should be understood as
\begin{align*}
    \sup_{y\in B\cap E_2} |TA_rf(y)| \le C_2 \mathcal{M}_{E_2}(|Tf|^2)^{\frac{1}{2}}(x) + C_2 Sf(x)
\end{align*}
for all $x\in B \cap E_2$.

We also mention that although the assumption $\mu(E_2)<\infty$ plays a crucial role in the proof; see \cite[Lemma~2.5, 2.6]{ACDH}, the operator norm $\|T\|_{L^p(E_1)\to L^p(E_2)}$ does not depend on the size of $E_2$, i.e. $\mu(E_2)$.
\end{remark}

\begin{proposition}\label{lemma_R2_lowenergy_main}
$\sup_{\alpha\ge 1} \|\mathcal{R}_{\alpha,1}^1\|_{p\to p} \lesssim 1$ for all $2<p<\infty$.
\end{proposition}

\begin{proof}[Proof of Lemma~\ref{lemma_R2_lowenergy_main}]

We want to apply Theorem~\ref{thm_ACDH}. To to so, we need to further decompose the $y$-space in a proper way. For each $\alpha \ge 1$, let $\{B_m^j = B_m\left(y_j, r_\alpha^{\beta+1}\right)\}_{j\in J}$ be a maximal $r_\alpha^{\beta+1}$-separated subset of $\mathbb{R}^m$ such that
\begin{align*}
    &(1) \mathbb{R}^m = \cup_j B_m^j,\quad (2) B_m \left(y_j,\frac{r_\alpha^{\beta+1}}{2} \right) \cap B_m\left(y_i, \frac{r_\alpha^{\beta+1}}{2}\right) = \emptyset,\\
    &(3) \sum_j \chi_{B_m^j}(y) \le C,\quad \forall y\in \mathbb{R}^m
\end{align*}
for some $C>0$ not depending on $r_\alpha$ (again, a consequence of volume doubling property); see for example \cite{Coulhon-Saloff} or \cite{CD2}. Let $\{\mathcal{Y}_j\}$ be a smooth partition of unity subordinate to this covering. To simplify the notations, we use $\mathcal{X}_{\alpha,j}(\xi)$ to denote the cut-off function $\mathcal{X}_\alpha(x) \mathcal{Y}_j(y)$.

Then, by the support of the kernel of $\mathcal{R}_{\alpha,1}^1$, one deduces
\begin{align*}
    \left|\mathcal{R}_{\alpha,1}^1f\right| \le \sum_j \left|\chi_{4B_n^\alpha \times 4B_m^j} \mathcal{R}_{\alpha,1} \left(\mathcal{X}_{\alpha,j}f\right)\right|.
\end{align*}
By finite overlap property, it suffices to show that 
\begin{align*}
\sup_{\alpha,j}\|\mathcal{R}_{\alpha,1}\|_{L^p\left(B_n^\alpha \times B_m^j\right)\to L^p\left(4B_n^\alpha \times 4B_m^j\right)}\lesssim 1.    
\end{align*}
Next, we set $A_{r} = e^{-r^2 L}$, $p_0=\infty$, $E_1 = B_n^\alpha \times B_m^j$, $E_2 = 4B_n^\alpha \times 4B_m^j$, and
\begin{align*}
    Tf(\xi) = \left|\int_0^{r_\alpha^2} \nabla_L e^{-tL}f(\xi) \frac{dt}{\sqrt{t}}\right|.
\end{align*}
We show that \eqref{thm_ACDH1}, \eqref{thm_ACDH2} hold for the above setting and the method is to employ the argument from \cite[Section~3.2, 4]{ACDH}; see also \cite[Lemma~4.3]{Gilles}.

Before doing that, we need to verify the $L^2$-boundedness of $\mathcal{R}_{\alpha,1}$. Indeed, by Lemma~\ref{lemma_R2_highenergy}, and the $L^2$-boundedness of the Riesz transform:
\begin{align*}
    \|\mathcal{R}_{\alpha,1}\|_{2\to 2} \le \|\mathcal{R}_{\alpha,2}\|_{2\to 2} + \|\mathcal{R}_\alpha\|_{2\to 2} \lesssim 1 + \|\mathcal{R}\|_{2\to 2}.
\end{align*}

From now on, let $B \subset \mathbb{R}^{n+m}$ be a geodesic ball with radius $r$, and $B\cap E_2 \ni \xi = (x,y)$. Suppose $f$ is a locally integrable function supported in $E_1$.

\begin{proof}[Proof of \eqref{thm_ACDH1}]
Write $f = f\chi_{2B} + f(1-\chi_{2B}):= f_1 + f_2$. By the $L^2$-boundedness of $T(I-A_r)$, we have by doubling condition
\begin{align*}
    \int_{B\cap E_2} |T(I-A_r)f_1|^2 \le \int_{2B \cap E_1} |f|^2 \le |2B \cap E_2| \mathcal{M}_{E_2}(|f|^2)(\xi) \le C |B\cap E_2| \mathcal{M}_{E_2}(|f|^2)(\xi).
\end{align*}
Next, we write
\begin{align}\label{eq1}
|T(I-A_r)f_2(z)| \le \int_{\mathbb{R}^{n+m}} |K_r(z,\eta)| |f_2(\eta)| d\eta,
\end{align}
where
\begin{align*}
    K_r(z,\eta) = \int_0^\infty |g_r(t)| |\chi_{E_2}(z) \nabla_L e^{-tL}(z,\eta) \chi_{E_1}(\eta)| dt,
\end{align*}
and
\begin{align*}
    g_r(t)=\frac{\chi_{[0,r_\alpha^2](t)}}{\sqrt{t}} - 
\frac{\chi_{[r^2,r_\alpha^2+r^2]}(t)}{\sqrt{t-r^2}}.
\end{align*}
Therefore, for $z \in E_2\cap B$ and $\eta \in E_1$,
\begin{align}\label{eq2}
    |K_r(z,\eta)|&\lesssim \int_0^{r_\alpha^2} \frac{|g_r(t)|}{\sqrt{t} V(z,\sqrt{t})}e^{-c\frac{d(z,\eta)^2}{t}} dt + \int_{r_\alpha^2}^{r^2+r_\alpha^2} \frac{|g_r(t)|}{r_\alpha V(z,\sqrt{t})} e^{-c\frac{d(z,\eta)^2}{t}} dt.  
\end{align}
From \cite[Lemma~2.1]{CD2} (see also \cite[Lemma~4.4]{Gilles}), we know that for any locally integrable function $v$, 
\begin{align}\label{eq3}
    \int_{d(z,\eta)\ge r} e^{-\frac{d(z,\eta)^2}{ct}} |v(\eta)| d\eta \lesssim V(z,\sqrt{t}) e^{-\frac{r^2}{ct}} \inf_{a\in B(z,r)}\mathcal{M}v(a).
\end{align}
Then, plug \eqref{eq2} into \eqref{eq1} and use \eqref{eq3} to deduce (recall that $\textrm{supp}(f_2) \subset (2B)^c$)
\begin{align*}
|T(I-A_r)f_2(z)| \lesssim \mathcal{M}f(\xi) \left(\int_0^{r_\alpha^2} |g_r(t)| e^{-\frac{r^2}{ct}} \frac{dt}{\sqrt{t}} + \int_{r_\alpha^2}^{r_\alpha^2+r^2} |g_r(t)| e^{-\frac{r^2}{ct}} \frac{dt}{r_\alpha} \right):= \mathcal{M}f(\xi) (I+II).
\end{align*}
We only need to verify that both $I$ and $II$ are uniformly bounded with respect to $\alpha,j,r$. 

\textit{Case 1. $r\le r_\alpha$.} Further split $I = \int_0^{r^2} + \int_{r^2}^{r_\alpha^2}$. For the first integral, we estimate $|g_r(t)|$ by $t^{-1/2}$. Following by a straightforward calculation, one yields that it is bounded by the finite integral: $\int_1^\infty e^{-s} ds/s$. As for the second integral, we bound $|g_r(t)|$ by $t^{-1/2} - (t-r^2)^{-1/2}$. Then, by a change of variables, one gets the uniform upper bound:
\begin{align*}
    \int_0^1 e^{-s} \left| 1 - \sqrt{\frac{1}{1-s}} \right| \frac{ds}{s}<\infty,
\end{align*}
which confirms that $|I|\lesssim 1$.

While for $II$, we have
\begin{align*}
    II \lesssim r_\alpha^{-1} \int_{r_\alpha^2}^{r_\alpha^2+r^2} e^{-\frac{r^2}{ct}} \frac{dt}{\sqrt{t-r^2}}\le r_\alpha^{-1} \int_0^{r_\alpha^2} \frac{ds}{\sqrt{s}} \lesssim 1.
\end{align*}

\textit{Case 2. $r\ge r_\alpha$.} Note that in this case, $ I\le \int_0^{r^2} e^{-\frac{r^2}{ct}} \frac{dt}{t} \lesssim 1$. Next, we split $II = \int_{r_\alpha^2}^{r^2} + \int_{r^2}^{r^2+r_\alpha^2}$. The first integral vanishes and the second one can be bounded by
\begin{align*}
    r_\alpha^{-1} \int_{r^2}^{r^2+r_\alpha^2} \frac{dt}{\sqrt{t-r^2}} \lesssim 1
\end{align*}
as desired.
\end{proof}

To complete the proof of Proposition~\ref{prop_R2}, we also need to prove \eqref{thm_ACDH2}.

\begin{proof}[Proof of \eqref{thm_ACDH2}]
By \cite[Theorem~1.1]{DS2}, the Grushin space satisfies Poincaré inequality ($\textrm{P}_2$) and hence a result of \cite{HK} (also see \cite[Section~3.2]{ACDH}) guarantees the following pointwise estimate:
\begin{align*}
|u(\xi) - u(\eta)| \le C d(\xi,\eta) \left(\mathcal{M}(|\nabla_L u|^2)^{\frac{1}{2}}(\xi) + \mathcal{M}(|\nabla_L u|^2)^{\frac{1}{2}}(\eta) \right).
\end{align*}
Set $h = \int_0^{r_\alpha^2} e^{-tL}f \frac{dt}{\sqrt{t}}$. Note that $TA_rf = \nabla_L e^{-r^2L}h$. Let $z,\xi\in B\cap E_2$, we have by stochastic completeness (see \cite[Theorem~6.1]{DS1})
\begin{align*}
    &\left|\nabla_L e^{-r^2L}h(z)\right| = \left|\nabla_L e^{-r^2L}\left(h(z) - h(\xi)\right)\right|  \le \int_{\mathbb{R}^{n+m}} \left|\nabla_L e^{-r^2L}(z,\eta)\right| |h(\eta) - h(\xi)| d\eta\\
    &\le \frac{C}{V(z,r)} \int_{\mathbb{R}^{n+m}} \left(\frac{1}{r} + \frac{1}{r_\alpha}\right) \textrm{exp}\left(-\frac{d(z,\eta)^2}{cr^2}\right) d(\xi,\eta) \left(\mathcal{M}(|\nabla_L h|^2)^{\frac{1}{2}}(\xi) + \mathcal{M}(|\nabla_L h|^2)^{\frac{1}{2}}(\eta) \right)d\eta.
\end{align*}

\textit{Case 1. $r\le r_\alpha$.} In this case, the above is bounded by
\begin{align*}
    \frac{1}{V(z,r)} \int_{\mathbb{R}^{n+m}} \frac{d(\xi,\eta)}{r} \textrm{exp}\left(-\frac{d(z,\eta)^2}{cr^2}\right) \left(\mathcal{M}(|\nabla_L h|^2)^{\frac{1}{2}}(\xi) + \mathcal{M}(|\nabla_L h|^2)^{\frac{1}{2}}(\eta) \right)d\eta.
\end{align*}
Now, since $\xi,z\in E_2\cap B$ and $\eta \in E_1$, it is easy to see that $d(\xi,\eta)\le d(\eta,z)+2r$. Hence, the term $\frac{d(\xi,\eta)}{r}$ can be absorbed into the exponential, i.e. the above can be estimated by
\begin{align*}
    \frac{1}{V(z,r)} \int_{\mathbb{R}^{n+m}} \textrm{exp}\left(-\frac{d(z,\eta)^2}{c' r^2}\right) &\mathcal{M}(|\nabla_L h|^2)^{\frac{1}{2}}(\eta) d\eta\\
    &+ \frac{\mathcal{M}(|\nabla_L h|^2)^{\frac{1}{2}}(\xi)}{V(z,r)} \int_{\mathbb{R}^{n+m}} \textrm{exp}\left(-\frac{d(z,\eta)^2}{c' r^2}\right) d\eta.
\end{align*}
By \cite[Lemma~2.1]{CD2}, the above is bounded by 
\begin{align*}
    \mathcal{M}\left(\mathcal{M}(|\nabla_L h|^2)^{\frac{1}{2}}\right)(\xi) + \mathcal{M}(|\nabla_L h|^2)^{\frac{1}{2}}(\xi) \le C \mathcal{M}(|\nabla_L h|^2)^{\frac{1}{2}}(\xi),
\end{align*}
where the last inequality follows from \cite{CR}; see also \cite[Section~3.2]{ACDH}. 

Let $\Tilde{B}\ni \xi$. Observe that $\nabla_L h = Tf$. The following estimate is immediate.
\begin{align}\label{eq5}
    \frac{1}{|\Tilde{B}|}\int_{\Tilde{B}\cap E_2} |\nabla_L h|^2 \le \frac{|\Tilde{B}\cap E_2|}{|\Tilde{B}|} \mathcal{M}_{E_2}(|Tf|^2)(\xi) \le \mathcal{M}_{E_2}(|Tf|^2)(\xi).
\end{align}
On the other hand, using the support property of $f$ (i.e. $\textrm{supp}(f)\subset E_1$) and the gradient estimate for the heat kernel, one obtains
\begin{align*}
    \int_{\Tilde{B}\setminus E_2}|\nabla_L h|^2 \le \int_{\Tilde{B}\setminus E_2} \left| \int_0^{r_\alpha^2} \int_{E_1} \frac{C}{ V(z,\sqrt{t})} \textrm{exp}\left(-\frac{d(z,\eta)^2}{ct}\right) |f(\eta)| d\eta \frac{dt}{t} 
 \right|^2 dz.
\end{align*}
Now, since $z\in E_2^c$ and $\eta \in E_1$, the inner integral is bounded by
\begin{align*}
    \frac{1}{V(z,\sqrt{t})}\int_{d(z,\eta)\ge 3r_\alpha} \textrm{exp}\left(-\frac{d(z,\eta)^2}{ct}\right) |f(\eta)| d\eta \le C e^{-\frac{r_\alpha^2}{ct}} \mathcal{M}f(z) \le C e^{-\frac{r_\alpha^2}{ct}} \mathcal{M}(|f|^2)^{\frac{1}{2}}(z).
\end{align*}
Consequently,
\begin{align}\label{eq4}
\int_{\Tilde{B}\setminus E_2}|\nabla_L h|^2 \le \int_{\Tilde{B}\setminus E_2} \mathcal{M}(|f|^2)(z) \left( \int_0^{r_\alpha^2} e^{-\frac{r_\alpha^2}{ct}} \frac{dt}{t}\right)^2 dz \le C |\Tilde{B}| \mathcal{M}\left(\mathcal{M}(|f|^2)\right)(\xi).
\end{align}
Combining \eqref{eq5} and \eqref{eq4}, we confirm for $r\le r_\alpha$, $z,\xi \in B\cap E_2$
\begin{align*}
    \left|TA_rf(z)\right| = \left|\nabla_L e^{-r^2L}h(z)\right|\lesssim  \mathcal{M}_{E_2}(|Tf|^2)^{\frac{1}{2}}(\xi) + \mathcal{M}\left(\mathcal{M}(|f|^2)\right)^{\frac{1}{2}}(\xi).
\end{align*}

\textit{Case 2. $r\ge r_\alpha$.} If this is the case, we directly bound $|\nabla_L e^{-r^2L}h(z)|$ by 
\begin{align*}
    r_\alpha^{-1} \int_{\mathbb{R}^{n+m}} \frac{1}{V(z,r)} \textrm{exp}\left(-\frac{d(z,\eta)^2}{r^2}\right) |h(\eta)| d\eta \le C r_\alpha^{-1} \mathcal{M}h(\xi).
\end{align*}

In summary, by setting $Sf = \mathcal{M}\left(\mathcal{M}(|f|^2)\right)^{\frac{1}{2}} +  r_\alpha^{-1} \mathcal{M}h$, we have verified that for any $\xi \in B\cap E_2$,
\begin{align*}
    \sup_{z\in B\cap E_2}\left|TA_rf(z)\right| \le C \left(\mathcal{M}_{E_2}(|Tf|^2)^{\frac{1}{2}}(\xi) + Sf(\xi)\right).
\end{align*}
The result follows by noting that for all $2<p<\infty$,
\begin{align*}
    \|Sf\|_p \lesssim \| \mathcal{M}\left(\mathcal{M}(|f|^2)\right)^{\frac{1}{2}}\|_p +  r_\alpha^{-1} \|\mathcal{M}h\|_p \lesssim \|f\|_p,
\end{align*}
where the last inequality follows by maximal theorem and $\|h\|_p \lesssim r_{\alpha}\|f\|_p$ (contractivity of the heat kernel).

This completes the proof of \eqref{thm_ACDH2}.
   
\end{proof}

By Theorem~\ref{thm_ACDH}, we have proved Proposition~\ref{lemma_R2_lowenergy_main}. 

\end{proof}

Proposition~\ref{prop_R2} follows by Lemma~\ref{lemma_R2_highenergy}, \ref{lemma_R2_lowenergy_error}, Proposition~\ref{lemma_R2_lowenergy_main} and letting $\delta \to 0$; see the remarks given above Lemma~\ref{lemma_R2_highenergy}

\subsection{Harmonic annihilation}

In this subsection, we prove Theorem~\ref{thm_RR}. By \cite[Theorem~8.1]{DS1} and duality (\cite[Proposition~2.1]{CD}), it is enough to prove \eqref{eq_RRp} for $1<p<2$. Let $\phi_\epsilon$ $(\epsilon>0)$ be a smooth function defined on $\mathbb{R}^n$ such that

$\bullet$ $\textrm{supp}(\phi_\epsilon) \subset B_n(0,4\epsilon/3)$,

$\bullet$ $\phi_\epsilon=1$ on $B_n(0,\epsilon)$,

$\bullet$ $\|\phi_\epsilon\|_\infty + \epsilon\|\nabla \phi_\epsilon\|_\infty \lesssim 1$.

Let $\kappa \ge 4$ be the parameter introduced as before. Then $\phi_{\kappa^{-1} |x'|}(x)$ is a smooth function supported in the regime: $\{(x,x')\in \mathbb{R}^n \times \mathbb{R}^n;4|x'|\ge 3\kappa |x|\}$. Next, we decompose the Riesz kernel as follows
\begin{align}\label{eq_Grushin_1}
    \mathcal{R}(\xi,\eta) = \int_0^\infty \phi_{\kappa^{-1}|x'|}(x)& \nabla_L e^{-tL}(\xi,\eta) \frac{dt}{\sqrt{\pi t}}\\ 
    &+ \int_0^\infty \left[1 - \phi_{\kappa^{-1}|x'|}(x) \right] \nabla_L e^{-tL}(\xi,\eta) \frac{dt}{\sqrt{\pi t}}:= \mathcal{O}_1(\xi,\eta) + \mathcal{O}_2(\xi,\eta).
\end{align}
Note that $\mathcal{O}_2$ is supported in the range
\begin{align*}
    \{(\xi,\eta)\in \mathbb{R}^{n+m}\times \mathbb{R}^{n+m} \setminus \{\xi=\eta\}; |x'|\le \kappa |x|\} \subset \mathcal{D}_1 \cup \mathcal{D}_2.
\end{align*}
Hence, $\mathcal{O}_2(\xi,\eta)$ is bounded by some constant multiple of $|\mathcal{R}_1(\xi,\eta)| + |\mathcal{R}_2(\xi,\eta)|$. Next, let $1<p<2$ and $f\in C_c^\infty(\mathbb{R}^{n+m})$, $g\in C_c^\infty(\mathbb{R}^{n}\setminus \{0\} \times \mathbb{R}^m)$ with $\|g\|_{p'}=1$. Note that $C_c^\infty(\mathbb{R}^{n}\setminus \{0\} \times \mathbb{R}^m)$ is dense in $L^p(\mathbb{R}^{n+m})$. By duality, it is enough to show
\begin{align*}
    \left| \langle L^{1/2}f, g \rangle \right| \le C \|\nabla_L f\|_p \|g\|_{p'},\quad \forall 1<p<2
\end{align*}
for some constant $C>0$ not depending on $g$.

By resolution to identity and the positivity and self-adjointness of $L$, one deduces
\begin{align*}
    \langle L^{1/2}f, g \rangle = \left\langle \int_0^\infty L e^{-tL}f \frac{dt}{\sqrt{\pi t}}, g \right\rangle &= \left\langle \nabla_L f, \int_0^\infty \nabla_L e^{-tL}g \frac{dt}{\sqrt{\pi t}} \right\rangle\\
    &= \left\langle \nabla_L f, \mathcal{O}_1g \right\rangle + \left\langle \nabla_L f, \mathcal{O}_2 g \right\rangle.
\end{align*}
By Proposition~\ref{prop_R1} and Proposition~\ref{prop_R2}, we have $\|\mathcal{R}_1+\mathcal{R}_2\|_{q\to q}\lesssim 1$ for all $q\in (2,\infty)$. Hence,
\begin{align*}
|\left\langle \nabla_L f, \mathcal{O}_2 g \right\rangle| \le \|\nabla_L f\|_p \|\mathcal{O}_2g\|_{p'} \le C \|\nabla_L f\|_p \|\mathcal{R}_1g+\mathcal{R}_2g\|_{p'} \le C \|\nabla_L f\|_p \|g\|_{p'}.
\end{align*}
Define bilinear form
\begin{equation*}
    \mathcal{B}(f,g) := \int_{\mathbb{R}^{n+m}} \nabla_L f(\xi) \cdot \int_{\mathbb{R}^{n+m}} \int_0^\infty \phi_{\kappa^{-1}|x'|}(x) \nabla_L e^{-tL}(\xi,\eta) g(\eta)\frac{dt}{\sqrt{\pi t}} d\eta d\xi.
\end{equation*}
By integration by parts,
\begin{align*}
    \mathcal{B}(f,g) &= \int_{\mathbb{R}^{n+m}} f(\xi) \int_{\mathbb{R}^{n+m}} \int_0^\infty \phi_{\kappa^{-1}|x'|}(x) L_{\xi} e^{-tL}(\xi,\eta) g(\eta) \frac{dt}{\sqrt{\pi t}} d\eta d\xi\\
    &- \int_{\mathbb{R}^{n+m}} f(\xi) \int_{\mathbb{R}^{n+m}} \int_0^\infty (\nabla_L)_x [\phi_{\kappa^{-1}|x'|}(x)] \cdot \nabla_L e^{-tL}(\xi,\eta) g(\eta) \frac{dt}{\sqrt{\pi t}} d\eta d\xi.
\end{align*}

\begin{lemma}\label{lemma_key2}
Under the assumptions of Theorem~\ref{thm_RR}, the following estimates hold:
\begin{align*}
    \int_0^\infty &|(\nabla_L)_x [\phi_{\kappa^{-1}|x'|}(x)] \cdot \nabla_L e^{-tL}(\xi,\eta)| + |\phi_{\kappa^{-1}|x'|}(x) \partial_t e^{-tL}(\xi,\eta)| \frac{dt}{\sqrt{t}} \\
    &\lesssim \begin{cases}
        |x'|^{-\mathcal{Q}-1} \chi_{|x|\le c_1|x'|}, & |y-y'|\le |x'|^{\beta+1},\\
        |x'|^{-2} |y-y'|^{-\frac{\mathcal{Q}-1}{\beta+1}} \chi_{c_2|x|\le |x'| \le c_3|x|} + |y-y'|^{-\frac{\mathcal{Q}+1}{\beta+1}}\chi_{|x|\le c_4|x'|}, & |y-y'|\ge |x'|^{\beta+1}
    \end{cases}
\end{align*}
for some constants $c_1,c_2,c_3,c_4>0$.
\end{lemma}

\begin{proof}[Proof of Lemma \ref{lemma_key2}]
Note that for $|x'|\ge \kappa |x|$, $|x-x'|\sim |x'|$. Hence
\begin{align*}
    d(\xi,\eta) \sim |x-x'| + \frac{|y-y'|}{(|x|+|x'|)^\beta + |y-y'|^{\frac{\beta}{\beta+1}}} \sim |x'| + \frac{|y-y'|}{|x'|^\beta + |y-y'|^{\frac{\beta}{\beta+1}}}:= \sigma,
\end{align*}
and
\begin{align}\label{eq_delta}
    \sigma \sim \begin{cases}
        |x'|, & |y-y'|\le |x'|^{\beta+1},\\
        |y-y'|^{\frac{1}{\beta+1}}, & |y-y'|\ge |x'|^{\beta+1}.
    \end{cases}
\end{align}
Moreover, for $x$ in the support of $\nabla_L \phi_{\kappa^{-1}|x'|}$, we have $|x|\sim |x'|$. 

By volume property and the gradient estimates of the heat kernel, it is plain that
\begin{align*}
    \int_0^{|x|^2} \left|\nabla_L [\phi_{\kappa^{-1}|x'|}(x)]\right| \left|\nabla_L e^{-tL}(\xi,\eta)\right| \frac{dt}{\sqrt{t}} &\lesssim \int_0^{|x|^2} \frac{\chi_{|x|\sim |x'|}}{|x'|} \frac{e^{-\frac{d(\xi,\eta)^2}{ct}}}{|x|^{m\beta}(\sqrt{t})^{n+m}} \frac{dt}{t}\\
    &\lesssim |x'|^{-1-m\beta} \chi_{|x|\sim |x'|} \int_0^{|x'|^2} t^{-\frac{n+m}{2}-1} e^{-\frac{\sigma^2}{ct}} dt\\
    &\lesssim |x'|^{-1-m\beta} \sigma^{-n-m} e^{-\frac{\sigma^2}{c|x'|^2}} \chi_{|x|\sim |x'|}\\
    &\lesssim_\epsilon |x'|^{-1-m\beta+\epsilon} \sigma^{-n-m-\epsilon} \chi_{|x|\sim |x'|},\quad \forall \epsilon \ge 0\\
    &\lesssim \sigma^{-\mathcal{Q}-1} \chi_{|x|\sim|x'|}.
\end{align*}
On the other hand,
\begin{align*}
    \int_{|x|^2}^\infty \left|\nabla_L [\phi_{\kappa^{-1}|x'|}(x)]\right| \left|\nabla_L e^{-tL}(\xi,\eta)\right| \frac{dt}{\sqrt{t}} &\lesssim \int_{|x|^2}^\infty \frac{\chi_{|x|\sim |x'|}}{|x'|} \frac{e^{-\frac{d(x,y)^2}{ct}}}{(\sqrt{t})^{\mathcal{Q}}} \frac{dt}{|x|\sqrt{t}}\\
    &\lesssim |x'|^{-2} \chi_{|x|\sim |x'|} \int_{|x'|^2}^\infty e^{-\frac{\sigma^2}{ct}} t^{-\frac{\mathcal{Q}+1}{2}} dt\\
    &\lesssim |x'|^{-2} \sigma^{-\mathcal{Q}+1} \chi_{|x|\sim|x'|}.
\end{align*}
Next, for the second part of the integral, \cite{Gri_heatkernel} guarantees the following estimate for the time-derivative of the heat kernel: 
\begin{equation*}
    \left|\partial_t e^{-tL}(\xi,\eta) \right| \le \frac{C}{t V(\eta,\sqrt{t})} e^{-\frac{d(\xi,\eta)^2}{ct}}, \quad \forall \xi,\eta\in \mathbb{R}^{n+m},\quad \forall t>0,
\end{equation*}
where we also use the doubling condition to replace the ball $V(\xi,\sqrt{t})$ by $V(\eta,\sqrt{t})$.

Therefore, it follows by a similar argument as before that
\begin{align*}
    \int_0^\infty \left|\phi_{\kappa^{-1}|x'|}(x)\right| \left|L_{\xi} e^{-tL}(\xi,\eta) \right| \frac{dt}{\sqrt{t}} &\lesssim  \int_0^\infty \frac{t^{-\frac{3}{2}} \chi_{|x|\le c|x'|}}{V(\eta,\sqrt{t})} e^{-\frac{d(\xi,\eta)^2}{ct}} dt\\
    &\lesssim \int_0^{|x'|^2} \frac{ \chi_{|x|\le c|x'|}}{|x'|^{m\beta}(\sqrt{t})^{n+m}} e^{-\frac{\sigma^2}{ct}} \frac{dt}{t^{3/2}}\\
    &+ \int_{|x'|^2}^\infty \frac{ \chi_{|x|\le c|x'|}}{(\sqrt{t})^{\mathcal{Q}}} e^{-\frac{\sigma^2}{ct}} \frac{dt}{t^{3/2}}\\
    &\lesssim |x'|^{-m\beta+\epsilon} \sigma^{-n-m-1-\epsilon} \chi_{|x|\le c|x'|} + \sigma^{-\mathcal{Q}-1} \chi_{|x|\le c|x'|},\quad \forall \epsilon\ge 0\\
    &\lesssim \sigma^{-\mathcal{Q}-1} \chi_{|x|\le c|x'|}.
\end{align*}
The result follows by using \eqref{eq_delta}.
\end{proof}

With Lemma~\ref{lemma_key2} in mind, we define operators:
\begin{gather*}
    \mathcal{T}_1: u \mapsto |x| \int_{B_n(0,c|x|)^c} |x'|^{-\mathcal{Q}-1} \int_{B_m\left(y,|x'|^{\beta+1}\right)} |f(x',y')|dy' dx',\\
    \mathcal{T}_2: u \mapsto |x| \int_{B_n(0,c|x|)^c} \int_{B_m\left(y,|x'|^{\beta+1}\right)^c} |y-y'|^{-\frac{\mathcal{Q}+1}{\beta+1}}|f(x',y')| dy' dx',\\
    \mathcal{T}_3: u \mapsto |x| \int_{c_1|x|\le |x'|\le c_2|x|} |x'|^{-2} \int_{B_m\left(y,|x'|^{\beta+1}\right)^c} |y-y'|^{-\frac{\mathcal{Q}-1}{\beta+1}}|f(x',y')| dy' dx'.
\end{gather*}
Then, by Lemma~\ref{lemma_key2}, the bilinear form is bounded by
\begin{align}\label{eq_final}
    \left|\mathcal{B}(f,g)\right|\le C \int_{\mathbb{R}^{n+m}} \frac{|f(\xi)|}{|x|} \left|(\mathcal{T}_1+ \mathcal{T}_2 + \mathcal{T}_3)g(\xi)\right| d\xi.
\end{align}

\begin{lemma}\label{lemma_G_3333}
Under the assumptions of Theorem~\ref{thm_RR}, the following statements hold:

(1) $\mathcal{T}_1$ and $\mathcal{T}_2$ are bounded on $L^q$ for all $1<q<\infty$.

(2) $\mathcal{T}_3$ is bounded on $L^q$ for $\frac{n}{n-1}<q<\infty$.
\end{lemma}

\begin{proof}[Proof of Lemma~\ref{lemma_G_3333}]
By dyadic decomposition, we write
\begin{align}\label{eq_sum}
    |\mathcal{T}_1f(\xi)| \le  |x| \sum_{j\ge 0} \int_{B_n(0,c2^{j+1}|x|) \setminus B_n(0,c2^j|x|)} |x'|^{-\mathcal{Q}-1} \int_{B_m(y,|x'|^{\beta+1})} |f(\eta)| d\eta.
\end{align}
For the $j$th term in the sum, one easily obtains upper bound: 
\begin{align*}
    2^{-j(\mathcal{Q}+1)}|x|^{-\mathcal{Q}-1} \int_{B_n \left(0,c2^{j+1}|x| \right) \times B_m \left(y,c^{\beta+1} 2^{(j+1)(\beta+1)}|x|^{\beta+1} \right)} |f(\eta)| d\eta.
\end{align*}
By Lemma~\ref{lemma_G_volume}, the above integral is bounded by $\int_{B\left((0,y), c2^{j+1}|x|\right)}|f|\le C \left(2^{j+1}|x|\right)^{\mathcal{Q}} \mathcal{M}f(\xi)$. This implies that the $j$th term in the sum of \eqref{eq_sum} can be estimated by
\begin{align*}
    C 2^{-j(\mathcal{Q}+1)} |x|^{-1} 2^{(j+1)\mathcal{Q}} \mathcal{M}f(\xi) = C 2^{-j} |x|^{-1} \mathcal{M}f(\xi).
\end{align*}
Plug this into \eqref{eq_sum}. The result of $\mathcal{T}_1$ follows by maximal theorem.

The argument for $\mathcal{T}_2$ is similar. By applying dyadic decomposition twice, it suffices to treat the double sum:
\begin{align}\label{eq_G_T2}
    |x|^{-\mathcal{Q}} \sum_{j\ge 0} \sum_{i\ge 0} 2^{-j\frac{\mathcal{Q}+1}{\beta+1}} 2^{-i(\mathcal{Q}+1)} \int_{B_n \left((0,2^{i+1}|x| \right) \times B_m \left( y, 2^{j+1} 2^{(i+1)(\beta+1)} |x|^{\beta+1} \right)} |f|.
\end{align}
Use Lemma~\ref{lemma_G_volume} again. Each $i,j\ge 0$ in the above sum is bounded by
\begin{align*}
    C\mathcal{M}f(\xi) 2^{-j\frac{\mathcal{Q}+1}{\beta+1}} 2^{\mathcal{Q}\frac{j+1}{\beta+1}} 2^{-i(\mathcal{Q}+1)} 2^{\mathcal{Q}(i+1)} |x|^{\mathcal{Q}} = C |x|^{\mathcal{Q}} \mathcal{M}f(\xi) 2^{-\frac{j}{\beta+1}} 2^{-i}.
\end{align*}
Combining this and \eqref{eq_G_T2} concludes the proof for $\mathcal{T}_2$.  

To this end, we also need to estimate $\mathcal{T}_3$. Note that in the kernel of $\mathcal{T}_3$, $|x| \sim |x'|$. Therefore, it is enough to verify the $L^p$-boundedness of the following operator:
\begin{align*}
    u \mapsto |x|^{-1} \int_{B_n(0,c|x|)} \int_{B_m \left( y, |x|^{\beta+1} \right)^c} |y-y'|^{-\frac{\mathcal{Q}-1}{\beta+1}}|u(\eta)| d\eta,
\end{align*}
which is bounded on $L^p$ for all $n'<p<\infty$; see Lemma~\ref{lemma_good2}.

This completes the proof of Lemma~\ref{lemma_G_3333}.
\end{proof}

Now, we are in a position to conclude the proof of Theorem~\ref{thm_RR}. By \eqref{eq_final} and Lemma~\ref{lemma_G_3333}, 
\begin{align*}
    |\mathcal{B}(f,g)| \lesssim \left\| \frac{f(\xi)}{|x|}\right\|_p \|g\|_{p'}.
\end{align*}
The result follows by \cite[Theorem~3.3]{D'ambrosio}, which asserts that on the Grushin space $\mathbb{R}^{n+m}$ with $n\ge 2$ and $m\ge 1$, the following Hardy's inequality holds:
\begin{align*}
    \int_{\mathbb{R}^{n+m}} \left|\frac{u(\xi)}{|x|}\right|^p d\xi \lesssim \int_{\mathbb{R}^{n+m}} |\nabla_L u(\xi)|^p d\xi,\quad \forall 1<p<n,
\end{align*}
for all $u\in C_c^\infty(\mathbb{R}^{n+m})$.

The proof of Theorem~\ref{thm_RR} is now complete.

\section{Riesz transform on Grushin spaces}

In this section, we turn our attention to the boundedness of the Riesz transform associated with the Grushin operator:
\begin{align*}
    L = \Delta_x + |x|^{2\beta} \Delta_y,\quad (x,y)\in \mathbb{R}^{n+m},
\end{align*}
i.e. $\nabla_L L^{-1/2}$. We show that the Riesz transform is bounded on $L^p$ for all $1<p<n$, which concludes the proof of Theorem~\ref{thm_RR_Grushin}.

\begin{theorem}\label{thm_R}
Let $n\ge 2$, $m\ge 1$ and $\beta > 0$. Then the Riesz transform associated to the Grushin operator $L$ given by \eqref{eq_grushin_operator} is bounded on $L^p$ for all $1<p<n$ in the sense:
\begin{align}\label{eq_Riesz}
\|\nabla_L L^{-1/2} f\|_p \le C \|f\|_p, \quad 1<p<n
\end{align}
for all $f\in C_c^\infty(\mathbb{R}^n \setminus \{0\} \times \mathbb{R}^m)$.
\end{theorem}

\begin{proof}[Proof of Theorem~\ref{thm_R}]
By \cite[Theorem~8.1]{DS1}, it is know that \eqref{eq_Riesz} holds for all $1<p\le 2$. Hence, it suffices to consider $2<p<n$ and $n\ge 3$. Recall the domain:
\begin{align*}
\mathcal{D}_3 = \{(\xi,\eta)\in \mathbb{R}^{n+m}\times \mathbb{R}^{n+m}; |x'| \ge \kappa |x| \}.
\end{align*}
By Proposition~\ref{prop_R1} and Proposition~\ref{prop_R2}, we only need to consider the operator restricted to the domain $\mathcal{D}_3$ with kernel:
\begin{align*}
    \mathcal{R}_3(\xi,\eta) = \int_0^\infty \nabla_L e^{-tL}(\xi,\eta) \chi_{\mathcal{D}_3}(\xi,\eta) \frac{dt}{\sqrt{\pi t}}.
\end{align*}
Note that for $(\xi,\eta) \in \mathcal{D}_3$, we have $|x-x'|\sim |x'|$ and
\begin{align*}
    \sigma:= d(\xi,\eta) \sim |x'| + \frac{|y-y'|}{|x'|^{\beta} + |y-y'|^{\frac{\beta}{\beta+1}}} \sim \begin{cases}
        |x'|, & |y-y'|\le |x'|^{\beta+1},\\
        |y-y'|^{\frac{1}{\beta+1}}, & |y-y'|\ge |x'|^{\beta+1}.
    \end{cases}
\end{align*}
By a straightforward computation as in Lemma~\ref{lemma_kernel_R1}, the kernel attains upper bound: for $\epsilon \ge 0$
\begin{align*}
    |\mathcal{R}_3(\xi,\eta)| \lesssim \chi_{|x|\le c|x'|} \begin{cases}
        |x|^{-m\beta+\epsilon} |x'|^{-n-m-\epsilon} + |x|^{-1} |x'|^{-\mathcal{Q}+1}, & |y-y'|\le |x'|^{\beta+1},\\
        |x|^{-m\beta+\epsilon} |y-y'|^{-\frac{n+m+\epsilon}{\beta+1}} + |x|^{-1} |y-y'|^{-\frac{\mathcal{Q}-1}{\beta+1}}, & |y-y'|\ge |x'|^{\beta+1}.
    \end{cases}
\end{align*}
Observe that by choosing $\epsilon = m\beta +1$, the terms $\chi_{|x|\le c|x'|} |x|^{-m\beta+\epsilon} |x'|^{-n-m-\epsilon} \chi_{|y-y'|\le |x'|^{\beta+1}}$ and $\chi_{|x|\le c|x'|} |x|^{-m\beta+\epsilon} |y-y'|^{-\frac{n+m+\epsilon}{\beta+1}} \chi_{|y-y'|\ge |x'|^{\beta+1}}$ are nothing but kernels of $\mathcal{T}_1$ and $\mathcal{T}_2$ defined in Lemma~\ref{lemma_G_3333} respectively. Hence, by Lemma~\ref{lemma_G_3333}, it is enough to estimate the following two operators:
\begin{align*}
\mathcal{T}_4: u \mapsto \int_{\mathbb{R}^{n+m}} \frac{\chi_{|x|\le c|x'|}(x') \chi_{|y-y'|\le |x'|^{\beta+1}}(y')}{|x| |x'|^{\mathcal{Q}-1}} u(x',y') d\eta,
\end{align*}
and
\begin{align*}
\mathcal{T}_5: u \mapsto \int_{\mathbb{R}^{n+m}} \frac{\chi_{|x|\le c|x'|}(x') \chi_{|y-y'|\ge |x'|^{\beta+1}}(y')}{|x| |y-y'|^{\frac{\mathcal{Q}-1}{\beta+1}}} u(x',y') d\eta.
\end{align*}
We treat $\mathcal{T}_4$ first. Let $1<p<n$. Let $\mathcal{A}(\xi,\eta)$ be a potential in the form:
\begin{align}\label{eq_potential}
    \mathcal{A} = |x|^{\alpha_1} |x'|^{\alpha_2} |y-y'|^{\alpha_3},
\end{align}
where $\alpha_1,\alpha_2,\alpha_3 \in \mathbb{R}$ satisfy conditions:
\begin{align*}
    &(1)\quad \alpha_1 + \alpha_2 + \alpha_3(\beta+1) = 1 - \frac{\mathcal{Q}}{p}, \quad (2)\quad 1-\frac{n}{p} < \alpha_1 < 0,\\
    &(3)\quad -\frac{m}{p} < \alpha_3 < \frac{m}{p'}.
\end{align*}
Choosing $\alpha_2$ appropriately, such parameters always exist.

Next, let $h\in C_c^\infty(\mathbb{R}^{n+m})$ with $\|h\|_{p'}\le 1$. We have by duality, and then followed by Hölder's inequality
\begin{align*}
    \|\mathcal{T}_4f\|_p &= \sup_{\|h\|_{p'}\le 1} \int_{\mathbb{R}^{n+m}} \int_{\mathbb{R}^{n+m}} \frac{\chi_{|x|\le c|x'|}(x') \chi_{|y-y'|\le |x'|^{\beta+1}}(y')}{|x| |x'|^{\mathcal{Q}-1}} f(\eta) h(\xi) d\eta d\xi\\
    &\le \sup_{\|h\|_{p'}\le 1} \left( \int_{\mathbb{R}^{n+m}} \int_{\mathbb{R}^{n+m}} \frac{|f(\eta)|^p \mathcal{A}^p}{|x|^p} \chi_{|x|\le c|x'|}(x) \chi_{|y-y'|\le |x'|^{\beta+1}}(y) d\eta d\xi \right)^{\frac{1}{p}}\\
    &\times \left( \int_{\mathbb{R}^{n+m}} \int_{\mathbb{R}^{n+m}} \frac{|h(\xi)|^{p'}}{|x'|^{p'(\mathcal{Q}-1)} \mathcal{A}^{p'} } \chi_{|x|\le c|x'|}(x') \chi_{|y-y'|\le |x'|^{\beta+1}}(y') d\eta d\xi \right)^{\frac{1}{p'}}.
\end{align*}
Note that a straightforward calculation yields that
\begin{align*}
    \int_{\mathbb{R}^{n+m}} \frac{\mathcal{A}^p}{|x|^p} &\chi_{|x|\le c|x'|}(x) \chi_{|y-y'|\le |x'|^{\beta+1}}(y) d\xi = \int_{B_n(0,c|x'|)} |x|^{p(\alpha_1-1)} \int_{B_m\left(y',|x'|^{\beta+1}\right)} |x'|^{\alpha_2 p} |y-y'|^{\alpha_3 p} d\xi\\
    &\sim |x'|^{\alpha_2 p}  \int_{B_n(0,c|x'|)} |x|^{p(\alpha_1-1)}  \int_0^{|x'|^{\beta+1}} s^{\alpha_3 p + m -1} ds dx, \quad \left[ \alpha_3 > -m/p\right]\\
    &\sim |x'|^{\alpha_2 p + m(\beta+1) + \alpha_3 p (\beta+1)}  \int_0^{|x'|} s^{\alpha_1 p - p + n -1} ds, \quad \left[ \alpha_1 > 1 -n/p \right]\\
    &\sim |x'|^{\mathcal{Q}-p+p\left(\alpha_1+\alpha_2+\alpha_3(\beta+1)\right)} = 1. \quad  \left[\alpha_1 + \alpha_2 + \alpha_3(\beta+1) = 1- \mathcal{Q}/p\right]
\end{align*}
On the other hand,
\begin{align*}
    \int_{\mathbb{R}^{n+m}} &\frac{1}{|x'|^{p'(\mathcal{Q}-1)} \mathcal{A}^{p'}} \chi_{|x|\le c|x'|}(x') \chi_{|y-y'|\le |x'|^{\beta+1}}(y') d\eta\\
    &= \int_{B_n(0,c'|x|)^c} \frac{|x|^{-\alpha_1 p'}}{|x'|^{p'(\mathcal{Q}-1)+\alpha_2 p'}} \int_{B_m(y,|x'|^{\beta+1})} |y-y'|^{-\alpha_3 p'} dy' dx'\\
    &\sim \int_{B_n(0,c'|x|)^c} \frac{|x|^{-\alpha_1 p'}}{|x'|^{p'(\mathcal{Q}-1)+\alpha_2 p'}} \int_0^{|x'|^{\beta+1}} s^{m-\alpha_3 p' -1} ds dx', \quad \left[ \alpha_3 < m/p' \right]   \\
    &\sim |x|^{-\alpha_1 p'} \int_{|x|}^\infty s^{-p'(\mathcal{Q}-1+\alpha_2) + m(\beta+1) - \alpha_3 p' (\beta+1) + n -1} ds, \quad \left[\alpha_1 < 0\right]\\
    &\sim |x|^{-p'(\mathcal{Q}-1+\alpha_2) + m(\beta+1) - \alpha_3 p' (\beta+1) + n - \alpha_1 p'} = 1. \quad \left[ \alpha_1 + \alpha_2 + \alpha_3(\beta+1) = 1- \mathcal{Q}/p\right]
\end{align*}
It then follows by Fubini's theorem,
\begin{align*}
    \|\mathcal{T}_4f\|_p &\lesssim \sup_{\|h\|_{p'}\le 1} \left( \int_{\mathbb{R}^{n+m}} |f(\eta)|^p \left( \int_{\mathbb{R}^{n+m}} \frac{\mathcal{A}^p}{|x|^p} \chi_{|x|\le c|x'|}(x) \chi_{|y-y'|\le |x'|^{\beta+1}}(y) d\xi    \right) d\eta \right)^{\frac{1}{p}} \\
    &\times \left( \int_{\mathbb{R}^{n+m}} |h(\xi)|^{p'} \left(\int_{\mathbb{R}^{n+m}} \frac{1}{|x'|^{p'(\mathcal{Q}-1)} \mathcal{A}^{p'}} \chi_{|x|\le c|x'|}(x') \chi_{|y-y'|\le |x'|^{\beta+1}}(y') d\eta\right)   d\xi \right)^{\frac{1}{p'}}\\
    &\lesssim \sup_{\|h\|_{p'}\le 1} \|f\|_p \|h\|_{p'} \le \|f\|_p
\end{align*}
as desired. This completes the proof of the $L^p$-boundedness of $\mathcal{T}_4$ for all $1<p<n$.

To this end, we also need to consider $\mathcal{T}_5$. Let $\mathcal{A}$ be the potential defined in \eqref{eq_potential}. This time, we assume the parameters satisfy the following conditions:
\begin{align*}
    &(1) \quad \alpha_1+\alpha_2+\alpha_3(\beta+1) = \frac{1}{p'}, \quad (2) \quad -\frac{n-1}{p'(\beta+1)} < \alpha_3 < \frac{n-1}{p(\beta+1)},\\
    &(3)\quad \alpha_1, \alpha_2 <0.
\end{align*}
Note that the existence of these parameters is guaranteed by the assumption $1<p<n$. Similarly, we assume $h\in C_c^\infty(\mathbb{R}^{n+m})$ with $\|h\|_{p'}\le 1$, and consider the bilinear form:
\begin{align*}
    \int_{\mathbb{R}^{n+m}} \int_{\mathbb{R}^{n+m}} \frac{\chi_{|x|\le c|x'|}(x') \chi_{|y-y'|\ge |x'|^{\beta+1}}(y')}{|x| |y-y'|^{\frac{\mathcal{Q}-1}{\beta+1}}} f(\eta) h(\xi) d\eta d\xi.
\end{align*}
By Hölder's inequality and Fubini's theorem, the above is bounded by
\begin{align*}
    &\left(    \int_{\mathbb{R}^{n+m}} |f(\eta)|^p \int_{\mathbb{R}^{n+m}} \frac{\mathcal{A}^p}{|x|^p |y-y'|^{\frac{\mathcal{Q}-1}{\beta+1}}} \chi_{|x|\le c|x'|}(x) \chi_{|y-y'|\ge |x'|^{\beta+1}}(y) d\xi d\eta             \right)^{\frac{1}{p}}\\
    &\times \left(    \int_{\mathbb{R}^{n+m}} |h(\xi)|^{p'}   \int_{\mathbb{R}^{n+m}} \frac{1}{\mathcal{A}^{p'} |y-y'|^{\frac{\mathcal{Q}-1}{\beta+1}}} \chi_{|x|\le c|x'|}(x') \chi_{|y-y'|\ge |x'|^{\beta+1}}(y')  d\eta d\xi          \right)^{\frac{1}{p'}},
\end{align*}
and we only need to show that both inner integrals are uniformly bounded. Indeed, a straightforward computation gives that
\begin{align*}
    \int_{\mathbb{R}^{n+m}} &\frac{\mathcal{A}^p}{|x|^p |y-y'|^{\frac{\mathcal{Q}-1}{\beta+1}}} \chi_{|x|\le c|x'|}(x) \chi_{|y-y'|\ge |x'|^{\beta+1}}(y) d\xi\\
    &\sim |x'|^{\alpha_2 p} \int_{B_n(0,c|x'|)} |x|^{\alpha_1 p - p} \int_{|x'|^{\beta+1}}^\infty s^{p\alpha_3 - \frac{\mathcal{Q}-1}{\beta+1}+m-1} ds dx, \quad \left[ \alpha_3 < (n-1)/\left(p(\beta+1)\right) \right]  \\
    &\sim |x'|^{\alpha_2 p} \int_{0}^{|x'|} s^{\alpha_1 p -p +p\alpha_3 (\beta+1) - \mathcal{Q}+1+m(\beta+1)+n-1} ds,   \quad  \left[\alpha_2 < 0 \right] \\
    &\sim |x'|^{p(\alpha_1+\alpha_2+\alpha_3(\beta+1))-p+1} = 1. \quad \left[  \alpha_1+\alpha_2+\alpha_3(\beta+1) = 1/p'\right]
\end{align*}
Meanwhile,
\begin{align*}
    &\int_{\mathbb{R}^{n+m}} \frac{1}{\mathcal{A}^{p'} |y-y'|^{\frac{\mathcal{Q}-1}{\beta+1}}} \chi_{|x|\le c|x'|}(x') \chi_{|y-y'|\ge |x'|^{\beta+1}}(y')  d\eta\\
    &\sim |x|^{-p' \alpha_1} \int_{B_n(0,c'|x|)^c} |x'|^{-p' \alpha_2} \int_{|x'|^{\beta+1}}^\infty s^{-p' \alpha_3 - \frac{\mathcal{Q}-1}{\beta+1} +m-1} ds dx',   \quad \left[\alpha_3 > - (n-1)/\left(p'(\beta+1)\right) \right]   \\
    &\sim |x|^{-p' \alpha_1} \int_{|x|}^\infty s^{-p' \alpha_2 + m(\beta+1) -\mathcal{Q}+1 - p' \alpha_3(\beta+1) +n-1} ds,    \quad  \left[\alpha_1 < 0 \right] \\
    &\sim |x|^{-p'(\alpha_1+\alpha_2+\alpha_3(\beta+1))+1} = 1.  \quad  \left[\alpha_1+\alpha_2+\alpha_3(\beta+1) = 1/p'\right]
\end{align*}
This completes the proof of the $L^p$-boundedness of $\mathcal{T}_5$ for all $1<p<n$ and hence Theorem~\ref{thm_R}.

\end{proof}

Theorem~\ref{thm_RR_Grushin} is now complete by combining Theorem~\ref{thm_RR} and Theorem~\ref{thm_R}.

\begin{remark}
We note that Theorem~\ref{thm_R} may not be optimal. Indeed, by \cite{DS3} (see also \cite{JST}), it is known that in certain special cases, \eqref{R_p} holds for all $1<p<\infty$ on Grushin spaces. For this reason, in this note we do not give an endpoint estimate for $p=n$. However, to the best of the author's knowledge, Theorem~\ref{thm_R} remains significant in the general setting of these spaces.

In fact, in a recent article \cite{HeSikora}, the authors establish the following inequality for the Hodge projector:
\begin{align}\label{Hodge}
    \| d \Delta^{-1} \delta \omega \|_p \le C \| (I - \mathcal{H}) \omega \|_p,
\end{align}
where $\mathcal{H}$ denotes the orthogonal projection onto the harmonic part of a $1$-form, for some range of $p$, say $1 < p < p_0$, within the framework of manifolds with ends.  
It is conjectured that such an estimate holds on a broader class of manifolds and metric spaces with sub-Riemannian structure.  
This inequality is particularly powerful in the sense that once the space of $L^2$ harmonic $1$-forms, $\mathcal{H}_{L^2}^1$, is known to be trivial, then \eqref{Hodge} self-improves to
\begin{equation}\label{Hodge2}
    \| d \Delta^{-1} \delta \omega \|_p \le C \| \omega \|_p
\end{equation}
for a much larger range of $p$ (by duality and interpolation).  
Consequently, the range of boundedness of \eqref{R_p} can be obtained by a standard argument via \eqref{Hodge2} and \eqref{eq_RRp} (see, for instance, \cite[Lemma~0.1]{AC}). We expect that Theorem~\ref{thm_R} can be significantly strengthened, but a more detailed analysis will
be presented in a forthcoming project.
\end{remark}

\bigskip

{\bf Acknowledgments.} Part of this article is included in the author's Ph.D. thesis. I want to express my gratitude to my supervisor Adam Sikora for all the support and encouragement.


\bibliographystyle{abbrv}

\bibliography{references.bib}

\end{document}